\documentclass[10pt,reqno]{amsart}
\usepackage{amsmath,amssymb,latexsym,soul,cite,mathrsfs}
\usepackage{color,enumitem,graphicx}
\usepackage[colorlinks=true,urlcolor=blue,
citecolor=red,linkcolor=blue,linktocpage,pdfpagelabels,
bookmarksnumbered,bookmarksopen]{hyperref}
\usepackage[english]{babel}
\usepackage{enumitem}

\usepackage[left=1.8cm,right=1.8cm,top=2.1cm,bottom=2.1cm]{geometry}

\usepackage[hyperpageref]{backref}

\makeatletter
\newcommand{\leqnomode}{\tagsleft@true}
\newcommand{\reqnomode}{\tagsleft@false}
\makeatother

\date{}
\def\nd{\noindent}
\def\thend{\rule{3mm}{3mm}}

\newtheorem{theorem}{Theorem}[section]

\newtheorem{Example}{Example}[section]
\newtheorem{prop}{Proposition}[section]
\newtheorem{lem}{Lemma}[section]
\newtheorem{rmk}{Remark}[section]

\newcommand{\Int}{\displaystyle\int_{\Rn}}

\newcommand{\Rn}{\mathbb{R}^N}

\begin{document}
	\title[Fractional elliptic systems via the nonlinear Rayleigh quotient]{Nonlocal elliptic systems via nonlinear Rayleigh quotient with general concave and coupling nonlinearities}
	\vspace{1cm}
	%%%%%%%%%%%%%%%%%%%%%%%%%%%%%%%%%%%%%%%%%%%%%%%%%%%%%%%%%%%%%%%%%%%%%%%%
	
	\author{Edcarlos D. Silva}
	\address{Edcarlos D da Silva \newline  Universidade Federal de Goias, IME, Goi\^ania-GO, Brazil}
	\email{\tt edcarlos@ufg.br}
	
	\author{Elaine A. F. Leite}
	\address{Elaine A. F. Leite \newline Intituto Federal de Goi\'as, Campus Goi\^anica, Goi\^ania-GO, Brazil }
	\email{\tt elainealtino@discente.ufg.br}
	
	\author{Maxwell L. da Silva}
	\address{SILVA, Maxwell L. \newline Universidade Federal de Goias, IME, Goi\^ania-GO, Brazil }
	\email{\tt maxwell@ufg.br}

	
	\thanks{The first author was partially supported by CNPq with grant 309026/2020-2. The second author was partially supported by IFG with grant 23744.001007/2019-69.}
	\begin{abstract}
		In this work, we shall investigate existence and multiplicity of solutions for a nonlocal elliptic systems driven by the fractional Laplacian. Specifically, we establish the existence of two positive solutions for following class of nonlocal elliptic systems:	
		\begin{equation*}
			\left\{\begin{array}{lll} 
				(-\Delta)^su +V_1(x)u =   \lambda|u|^{p - 2}u+  \frac{\alpha}{\alpha+\beta}\theta |u|^{\alpha - 2}u|v|^{\beta}, \;\;\; \mbox{in}\;\;\; \mathbb{R}^N,\\
				(-\Delta)^sv +V_2(x)v=   \lambda|v|^{q - 2}v+  \frac{\beta}{\alpha+\beta}\theta |u|^{\alpha}|v|^{\beta-2}v,  \;\;\; \mbox{in}\;\;\; \mathbb{R}^N,\\
				(u, v) \in H^s(\mathbb{R}^N) \times  H^s(\mathbb{R}^N).
			\end{array}\right.
		\end{equation*}
		Here we mention that $\alpha > 1, \beta > 1, 1 \leq p \leq q < 2 < \alpha + \beta < 2^*_s$, $\theta > 0, \lambda > 0, N > 2s$, and $s \in (0,1)$. Notice also that continuous potentials $V_1, V_2: \mathbb{R}^N \to \mathbb{R}$ satisfy some extra assumptions. Furthermore, we find the largest positive number $\lambda^* > 0$ such that our main problem admits at least two positive solutions for each $ \lambda \in (0, \lambda^*)$. This can be done by using the nonlinear Rayleigh quotient together with the Nehari method. The main feature here is to minimize the energy functional in Nehari manifold which allows us to prove our main results without any restriction on size of parameter $\theta > 0$.
	\end{abstract}
	\maketitle
	
	\section{Introduction}
	In this research our main objective is to investigate the existence and multiplicity of positive solutions for a class of nonlocal elliptic systems defined in the whole space $\mathbb{R}^N$ where the coupled term is superlinear. Furthermore, we consider different kind of concave convex terms in our main problem. More specifically, we shall consider the following nonlocal elliptic system:
	\begin{equation}\left\{\begin{array}{lll}\label{sistema Principal} 
			(-\Delta)^su +V_1(x)u =   \lambda|u|^{p - 2}u+  \frac{\alpha}{\alpha+\beta}\theta |u|^{\alpha - 2}u|v|^{\beta},  \;\;\; \mbox{in}\;\;\; \mathbb{R}^N, \\
			(-\Delta)^sv +V_2(x)v=   \lambda|v|^{q - 2}v+  \frac{\beta}{\alpha+\beta}\theta |u|^{\alpha}|v|^{\beta-2}v,  \;\;\; \mbox{in}\;\;\; \mathbb{R}^N, \\
			(u, v) \in H^s(\mathbb{R}^N) \times  H^s(\mathbb{R}^N).
		\end{array}\right.
	\end{equation}
	Recall that $(-\Delta)^s$ represents the fractional Laplacian, $\alpha > 1, \beta > 1, 1 \leq p \leq q < 2 < \alpha + \beta < 2^*_s$, $\theta > 0, \lambda > 0, N > 2s$, and $s \in (0,1), 2_s^* = 2 N/(N - 2s)$. Later on, we shall consider some extra assumptions on the continuous potentials $V_1, V_2: \mathbb{R}^N \to \mathbb{R}$.
	
	For the scalar case, we observe there are several physical applications such as nonlinear optics. Furthermore, the fractional Laplacian operator has been accepted as a model for diverse physical phenomena such as diffusion-reaction equations, quasi-geostrophic theory, Schr\"odinger equations, Porous medium problems, see for instance \cite{aka,biboa,ya,consta,pablo,las}. For further applications such as continuum mechanics, phase transition phenomena, populations dynamics, image processes, game theory, see \cite{ber,cafa,las}. It is important to stress that  semilinear nonlocal reaction-diffusion equations have attracted some attention in the last few years. The main motivation for this kind of problem is to combine nonlinear and quasilinear nonlocal terms in order to model a nonlinear diffusion. On this subject we refer the reader to \cite{vasquez1,vasquez2,vasquez3} and references therein.

	Now, we mention also that nonlocal elliptic systems have been widely studied in the last years, see \cite{brande,colorado,jotinha}. This kind of problem deals with the fractional Laplacian operator which was generalized in many context, see \cite{guia}.  For the local case, that is, assuming that $s = 1$ we refer the reader to the important works \cite{ambro1,ambro2,defigueiredo,maia,oliveira,pompo}. In those works was proved several results on existence and multiplicity of solutions taking into account some hypotheses on the potential as well as in the nonlinearity. For further results on fractional elliptic system we refer the interested reader to \cite{guo,lu,Laskin1}.  It is important to recall that a pair $(u,v)$ is said to be a ground state solution for the System  \eqref{sistema Principal} when $(u,v)$ has the minimal energy among all nontrivial solutions. At the same time, a nontrivial solution $(u,v)$ is a bound solution for the System  \eqref{sistema Principal} whenever $(u,v)$ has finite energy.
	
	Recall that in \cite{colorado} the authors considered the following nonlocal elliptic system
	\begin{equation}\label{lindo}
		\left\{\begin{array}{lll}
			(-\Delta)^su + \lambda_1 u =   \mu_1|u|^{2p - 2}u +  \beta |u|^{p - 2}u |v|^{p},  \;\;\; \mbox{in}\;\;\; \mathbb{R}^N, \\
			(-\Delta)^sv + \lambda_2 v= \mu_2 |v|^{2 p - 2}v + \beta |u|^{p}|v|^{p-2}v,  \;\;\; \mbox{in}\;\;\; \mathbb{R}^N, \\
			(u, v) \in H^s(\mathbb{R}^N) \times  H^s(\mathbb{R}^N).
		\end{array}\right.
	\end{equation}
	where $N \in \{1, 2,3\}$, $\lambda_i, \mu_i > 0, i = 1,2, p \geq 2, (p-2)N/p < s < 1$. The authors proved several results on existence of ground and bound state solutions assuming that $p > 2$ or $p = 2$. The main ingredient in that work was to combine the Nehari method and the size of $\beta > 0$ in order to avoid semitrivial solutions. In fact, assuming that $\beta > 0$ is small, the authors proved also that the bound state for the Problem \eqref{lindo} is a semitrivial solution. On the other hand, assuming that $\beta > 0$ is large enough, the authors proved existence of ground state solutions $(u,v)$ for the Problem \eqref{lindo} where $u \neq 0$ and $v \neq 0$. In other words, for each $\beta > 0$ large enough, the authors ensured that for the Problem \eqref{lindo} there exists at least one non-semitrivial ground state solution.  At the same time, we observe that Problem \eqref{lindo} is superlinear at the origin and at infinity.  Motivated in part by the previous discussion we shall consider existence and multiplicity of solutions for the System  \eqref{sistema Principal} assuming that the nonlinear term is a concave-convex function. Furthermore, for the coupling term we consider a more general function due the fact that $1 < \alpha, \beta < 2^*_s$ where $\alpha$ and $\beta$ can be different. Hence, our main objective in the present work is to guarantee existence and multiplicity of solutions without any restriction on the size of $\theta$. For similar results on nonlocal elliptic problems we refer the reader also to \cite{ai,chen}.
	
	The main difficulty in the present work is to find the largest positive number $\lambda^* > 0$ in such way the Nehari method can be applied. It is well known that the fibering maps associated to the energy functional for the System \eqref{sistema Principal} has some inflections points for each $\lambda > 0$ large enough. This kind of problem does not allow us to apply the Lagrange Multiplier Theorem in general. Since we are looking for a minimization problem for the energy functional associated to the System  \eqref{sistema Principal} in the Nehari manifold we need to prove that any minimizer are nondegenerated.  It is important to stress that for each $\lambda \in (0, \lambda^*)$ the fibering maps does not admit any inflections points. Therefore, we employ the nonlinear Rayleigh quotient proving existence of minimizers for the energy functional in some subsets of the Nehari manifold. More precisely, we find the existence of at least two positive solutions for the System  \eqref{sistema Principal} whenever $\lambda \in (0, \lambda^*)$. On this subject, we refer the interested reader to the important works \cite{yavdat1,yavdat2}. Another difficulty in the present work is to treat the energy functional associated to the System  \eqref{sistema Principal} taking into account the nonlinear Raleigh quotient where $p \leq q$ and $1 < p, q < 2$. In fact, for each $p \neq q$ the functionals associated to the nonlinear Raleigh quotient are given only implicitly. This difficulty is inherent with the complexity of the System  \eqref{sistema Principal}. Indeed, our main problem consider a huge class of concave terms due to the fact that $p$ and $q$ can be different. Furthermore, for the coupled term we have also some difficulties. The first one arises from the fact that given a pair $ (u,v) \neq (0,0)$ the projections on the Nehari manifold is not defined in general. Under these conditions, we introduce an open set $\mathcal{A}$ in such way that any function $(u,v) \in \mathcal{A}$ admits exactly two projections in the Nehari manifold.  The second difficulty appears due to the coupled term allow us to find semitrivial solutions of the type $(u,0)$ and $(0,v)$ for the System  \eqref{sistema Principal}. Here we refer the interested reader to the works \cite{ambro1,ambro2} where the authors considered some assumptions in order to avoid semitrivial solutions. It is worthwhile to mention that finding semitrivial solutions of the type $(u,0)$ and $(0,v)$ for the System  \eqref{sistema Principal} is equivalent to find respectively weak solutions for the following scalar problems:
	\begin{eqnarray}\label{eqcasopartic} 
		(-\Delta)^s u +V_1(x)u &=&   \lambda|u|^{p - 2}u, \, \, u \in H^s(\mathbb{R}^N)
	\end{eqnarray}
	and 
	\begin{eqnarray}\label{eqcasopartic} 
		(-\Delta)^s v +V_2(x)v &=&   \lambda|v|^{q - 2}v, \, \, v \in H^s(\mathbb{R}^N).
	\end{eqnarray}
	For this kind of scalar nonlocal elliptic problems we refer the reader to \cite{bisci,servadei, servadei1,servadei2,secchi,secchii,guia} and references therein. Here we mention that the literature for fractional elliptic problems is interesting and quite huge. For more references we infer the reader to \cite{chang,felmer,pala}.
	
	Our main objective in the present work is to ensure existence and multiplicity of solutions for the System  \eqref{sistema Principal} avoiding semitrivial solutions. More precisely, we prove our main results using the Nehari method together with the Nonlinear Rayleigh quotient which permit us to use find curve $s \mapsto G(s), s \in (-\delta, \delta)$ such that $G(s)$ belongs to the Nehari manifold for $\delta > 0$ small, see Proposition \ref{soluçao N^+} ahead.  As a consequence, by using a contradiction argument, we prove that any minimizer for the energy functional restricted to the Nehari manifold is not a semitrivial solution. Hence, we prove our main results without any kind of restriction on the size of parameter $\theta > 0$. In other words, the System  \eqref{sistema Principal} admits existence and multiplicity of solutions which are not semitrivial for each $\theta > 0$. To the best our knowledge, up to now, the present work is the first one proving existence and multiplicity of solutions for the System \eqref{sistema Principal} where the coupling term is superlinear and the concave terms are distinct. Once again we are able to prove that existence of solutions which are not semitrivial due to the behavior of the fibering maps together a fine analysis on the energy functional associated to the System  \eqref{sistema Principal}.

	\section{Assumptions and main theorems}
	As was mentioned in the introduction we shall investigate the existence of positive nontrivial weak solutions $(u, v)$ for the System \eqref{sistema Principal} such that $$\Int \vert u\vert^\alpha \vert v\vert^\beta dx >0$$ which can be done looking for the parameters $\lambda > 0$ and $\theta > 0$. The main objective is to ensure sharp conditions on the parameters $\lambda > 0$ such that the Nehari method can be applied using the nonlinear Rayleigh quotient. Throughout this work, we consider the following set of assumptions:
	
	\begin{itemize}
		\item[$(P)$] It holds $\alpha >1, \beta> 1, 1 \le p \leq q < 2 < \alpha + \beta  < 2^*_s, \theta > 0, \lambda > 0, N > 2s$ and $s \in (0,1), 2^*_s = 2 N/(N - 2 s)$;
		\item [$(V_0)$] The potential $V_1, V_2:\mathbb{R}^N\to \mathbb{R}$ is continuous function and there exists a constant $V_0 > 0$ such that $$V_1(x)\geq V_0 \,\ \ \ \, V_2(x)\geq V_0, \,\ \ \ \, \mbox{for all} \,\, x \in \mathbb{R}^N;$$
		\item [$(V_1)$] Its holds $\frac{1}{V_1} \in  L^1(\mathbb{R}^N)$  and $\frac{1}{V_2} \in L^1(\mathbb{R}^N).$
	\end{itemize}
	
	\begin{Example}
		Recall that the potentials $V_j: \mathbb{R}^N \to \mathbb{R}$, $j = 1, 2$ given by
		$V_j (x) = (1 + x^2)^{\gamma_j}, \gamma_j > N/2$ satisfy our assumption $(V_0)$ and $(V_1)$. In fact, we observe that $V_j(x) \geq 1$ holds for each $x \in \mathbb{R}^N$. Moreover, by using the Co-area Formulae, we show that $\frac{1}{V_j} \in L^1(\mathbb{R}^N), j = 1,2$. Indeed, we mention that 
		\begin{eqnarray*}
			\int\limits_{\mathbb{R}^N}\frac{1}{V_j}dx & = & \int\limits_{\mathbb{R}^N}\frac{1}{(1 + x^2)^{\gamma_j}} dx = C\int_{0}^{\infty}\frac{r^{N - 1}}{(1 + r^2)^{\gamma_j}} dr\\
			&\le& C\int_{0}^{1}\frac{r^{N - 1}}{(1 + r^2)^{\gamma_j}} dr + C\int_{1}^{\infty}r^{N - 1 - 2 \gamma_j} dr  \le C + C \left.r^{N - 2\gamma_j}\right|_{1}^{\infty} < \infty.
		\end{eqnarray*}
	\end{Example}
	At this stage, for each $s \in (0, 1)$,  we mention that the fractional Sobolev space is given by
	$$ H^s(\mathbb{R}^N)=\left \{   u:\mathbb{R}^N \to \mathbb{R};\int_{\mathbb{R}^N}\int_{\mathbb{R}^N}\frac{|u(x)-u(y)|^2}{|x-y|^{N+2s}} dxdy<\infty,
	\int_{\mathbb{R}^N}(-\Delta)^su\varphi dx = \int_{\mathbb{R}^N} u(-\Delta)^s\varphi dx, \\ \forall \varphi \in C^{\infty}_c(\mathbb{R}^N) \right\}. \\
	$$
	Furthermore, for each $s \in (0, 1)$, the fractional Laplacian operator of a measurable function $u: \mathbb{R}^N \to \mathbb{R}$ may be defined by in the following way 
	$$(-\Delta)^su(x) = -\frac{1}{2}C(N,s) \Int\frac{u(x + y) + u(x - y) - 2u(x)}{\vert u \vert^{N + 2s}}dy$$ for all $x \in \mathbb{R}^N$, where $C(N, s) = \left(\Int\frac{1 - cos(\xi)}{\vert \xi \vert ^{N + 2s}}d\xi\right)^{-1}.$
	The working space is defined by $X = X_1 \times X_2$ where
	$$X_1=\left \{u \in H^s(\mathbb{R}^N); \int _{\mathbb{R}^N} V_1(x)u^2dx<\infty \right\}\;\;\; \mbox{and}\;\;\; X_2= \left \{v \in H^s(\mathbb{R}^N); \int _{\mathbb{R}^N} V_2(x)v^2dx<\infty \right\}.$$
	Now, we consider the following set
	\begin{equation} \label{A}
		\mathcal{A}=\left \{ (u, v) \in X; \int_{\mathbb{R}^N} \vert u \vert^\alpha \vert v\vert^\beta dx > 0 \right\}.
	\end{equation}
	It is important to emphasize that $X$ is a Hilbert space endowed with the norm and inner product as follows:
	$$ \Vert(u,v)\Vert ^2=[(u, v)]^2 +\int_{\mathbb{R}^N} V_1(x) u^2 +V_2(x) v^2 dx, (u, v) \in X.$$
	and 
	\begin{eqnarray*}
		\left<(u,v),(\varphi, \psi)\right> & = &\int_{\mathbb{R}^N} \int_{\mathbb{R}^N} \frac{[u(x)-u(y)] [\varphi(x) - \varphi(y)]}{\vert x-y\vert^{N+2s}} dxdy\\ 
		& & + \int_{\mathbb{R}^N} \int_{\mathbb{R}^N} \frac{[v(x)-v(y) ] [\psi(x) - \psi(y)]}{\vert x-y \vert^{N+2s}} dxdy + \int_{\mathbb{R^N}} V_1 u \varphi +V_2  v \psi dx, (u,v) , (\phi, \psi) \in X.
	\end{eqnarray*}
	Throughout this work, we define the term $[(u, v)]^2$ as the Gagliardo semi-norm of the function $(u, v)$ which can be written as follows:
	$$ [(u, v)]^2=[u]^2+[v]^2, (u, v) \in X.$$
	Similarly, the Gagliardo semi-norm for the function $u$ is denoted by
	$$ [u]^2=\int_{\mathbb{R}^N}\int_{\mathbb{R}^N}\frac{[u(x)-u(y)]^2}{|x-y|^{N+2s}} dxdy.$$
	At this stage, we shall use a important tool in order to prove our main results. 
	\begin{lem} \label{chave}
		Let ($V_0$), ($V_1$), and $s \in (0, 1)$ such that $2s < N$. Then there exists a positive constant $C = C(n, p, s)$ such that for all $(u, v) \in X$ we obtain
		$$\Vert (u, v) \Vert_{r_1 \times r_2} \le C \Vert (u, v) \Vert$$
		holds for all $r_1, r_2 \in [1, 2^*_s]$.
		In other words, $X$ is continuously embedded into $L^{r_1}(\mathbb{R}^N) \times L^{r_2}(\mathbb{R}^N)$ for all $r_1, r_2 \in [1, 2^*_s]$. Furthermore, for each $(u_k, v_k)$ bounded sequence in $X$, up to a subsequence, there holds $(u_k, v_k) \to (u, v)$ in $L^{r_1}(\mathbb{R}^N) \times L^{r_2}(\mathbb{R}^N)$ for all $r_1, r_2 \in [1, 2^*_s)$. Hence, $X$ is compactly embedded into $L^{r_1}(\mathbb{R}^N) \times L^{r_2}(\mathbb{R}^N)$ with $r_1, r_2 \in [1, 2^*_s)$.
	\end{lem}
	
	\begin{rmk}
		It is important to stress that using more general hypothesis on the potential $V$ we can recovery the compactness of X into the Lebesgue spaces $L^r(\mathbb{R}^N)$ with $r \in [2, 2^*_s]$ with $s \in (0,1)$ and $2s < N$. Namely, we can consider also coercive potentials $V$, that is, we assume that 
		\begin{equation}
			\lim_{|x| \to \infty} V(x) = + \infty.
		\end{equation}
		For more general results on this subject we refer the reader to \cite{wang}.
	\end{rmk}
	The energy functional $E_{_\lambda}: X \to \mathbb{R}$ associated with System (\ref{sistema Principal}) is defined as follows:
	\begin{equation}\label{Ener} 
		E_{_\lambda}(u,v)=\frac{1}{2}\Vert(u,v)\Vert ^2 - \frac{\lambda}{p}{\Vert u \Vert}_{p}^{p}-\frac{\lambda}{q}{\Vert v \Vert}_{q}^{q} -\frac{\theta}{\alpha + \beta}\int_{\mathbb{R}^N} \vert u \vert^\alpha \vert v\vert^\beta dx \ \ (u,v)\in X = X_1 \times X_2.
	\end{equation}
	Under our hypotheses, we observe that $E_{\lambda}$ belongs to $C^1(X, \mathbb{R})$. Moreover, the pair $(u, v) \in X$ is a critical point for the functional $E_{\lambda}$ if, and only if, $(u, v)$ is a weak solution to the elliptic System  (\ref{sistema Principal}). Furthermore, the Gateaux derivative for $E_{\lambda}$ is given by
	\begin{eqnarray}
		E'_{_\lambda}(u,v)(\varphi, \psi) & = & \left<(u,v),(\varphi, \psi)\right> - \lambda\int_{\mathbb{R}^N} |u|^{p - 2}u\varphi dx -\lambda \int_{\mathbb{R}^N} \vert v\vert^{q-2}v\psi dx \nonumber\\
		&&-\frac{\theta}{\alpha + \beta}\alpha \int_{\mathbb{R}^N} \vert u \vert^{\alpha - 2}u\varphi \vert v \vert^\beta dx - \frac{\theta}{\alpha + \beta}\beta \int_{\mathbb{R}^N} \vert u \vert^{\alpha} \vert v \vert^{\beta-2}v \psi dx, (\varphi, \psi) \in X.
	\end{eqnarray}
	In particular, we obtain that 
	\begin{equation}\label{der ener} 
		E'_{_\lambda}(u,v)(u,v) = \Vert(u,v)\Vert ^2- \lambda \Vert u \Vert _{p} ^{p} - \lambda \Vert v \Vert _{q} ^{q} -\theta \int_{\mathbb{R^N}} |u|^\alpha |v|^\beta dx, \, \, (u, v ) \in X.
	\end{equation}
	Moreover, we observe that 
	\begin{equation}\label{der seg ener} 
		E''_{_\lambda}(u,v)(u,v)^2 = \Vert(u,v)\Vert ^2- \lambda (p -1) \Vert u \Vert _{p} ^{p} - \lambda (q - 1) \Vert v \Vert _{q} ^{q} -\theta (\alpha + \beta)\int_{\mathbb{R^N}} |u|^\alpha |v|^\beta dx, \,\,  (u,v) \in X.
	\end{equation}

	Now, we shall consider the Nehari method for our main System  (\ref{sistema Principal}) as follows:
	\begin{eqnarray}\label{Nehari}
		\mathcal {N}_{\lambda}& = & \left\{ (u, v) \in X\backslash {\{(0, 0)\}}; \lambda\Vert u \Vert_p^p + \lambda \Vert v \Vert_q^q = \Vert(u, v)\Vert^2 - \int_{\mathbb{R}^N} \vert u \vert^\alpha \vert v\vert^\beta dx \right\}.
	\end{eqnarray}
	Hence, we can split the Nehari manifold $\mathcal{N}_{\lambda}$ into three disjoint subsets in the following way:
	\begin{equation}\label{N+} 
		\mathcal {N}^+_{\lambda} = \{(u, v) \in \mathcal {N}_{\lambda}; E''_{\lambda}(u, v)(u, v)^2 > 0\},
	\end{equation}
	\begin{equation}\label{N-} 
		\mathcal {N}^-_{\lambda} = \{(u, v) \in \mathcal {N}_{\lambda}; E''_{\lambda}(u, v)(u, v)^2 < 0\},
	\end{equation}
	\begin{equation}\label{N0} 
		\mathcal {N}^0_{\lambda} = \{(u, v) \in \mathcal {N}_{\lambda}; E''_{\lambda}(u, v)(u, v)^2 = 0\}.
	\end{equation}
	The main objective in the present work is to find weak solutions for our main problem using for the following minimization problems:
	\begin{equation}\label{CN-} 
		C_{\mathcal{N}^-_{\lambda} \cap \mathcal{A}} = \inf_{(u, v) \in \mathcal{N}^-_{\lambda} \cap \mathcal{A}}E_{_\lambda}(u, v),
	\end{equation}
	\begin{equation}\label{CN+} 
		C_{\mathcal{N}^+_{\lambda} \cap \mathcal{A}} = \inf_{(u, v) \in \mathcal{N}^+_{\lambda} \cap \mathcal{A}}E_{_\lambda}(u, v).
	\end{equation}
	In other words, we shall prove that $C_{\mathcal{N}^-_{\lambda} \cap \mathcal{A} }$ and $C_{\mathcal{N}^+_{\lambda} \cap \mathcal{A}}$ are attained.	
	It is important to stress that
	$$(u, v) \in \mathcal{N}_{\lambda} \ \ \Leftrightarrow \ \  \lambda =  \frac{\Vert(u,v)\Vert ^2 -\theta  \int_\mathbb{R^N} |u|^\alpha |v|^\beta dx}{\Vert u \Vert _{p} ^{p} + \Vert v \Vert _{q} ^{q}}$$
	and  
	$$ E_{_\lambda}(u,v)=0  \ \ \Leftrightarrow \ \       \lambda = \frac{\frac{1}{2}\Vert(u,v)\Vert ^2-\frac{\theta}{\alpha+\beta}\int_{\mathbb{R}^N} |u|^\alpha|v|^\beta dx}{ \frac{1}{p}{\Vert u \Vert}_{p}^{p}+\frac{1}{q}{\Vert v \Vert}_{q}^{q}}.$$ 
	
	\begin{rmk}\label{N- cont A}
		Let $\mathcal{N}^-_{\lambda}$ and $\mathcal{A}$ be defined by \eqref{Nehari} and \eqref{A}, respectively. As a consequence, we obtain that 
		$\mathcal{N}^-_{\lambda} \subset \mathcal{A}.$
		It is straightforward to verify the last assertion arguing by contradiction. Suppose there exists $(u, v) \in \mathcal{N}^{-}_{\lambda}$ such that $(u, v) \notin \mathcal{A}$. In other words, $\int_{\mathbb{R^N}} |u|^\alpha |v|^\beta dx = 0$. Hence, we obtain that 
		$$E''_{_{\lambda}}(u, v)(u, v) = 2\Vert(u,v)\Vert ^2- \lambda p \Vert u \Vert _{p} ^{p} - \lambda q\Vert v \Vert _{q} ^{q} < 0 $$
		However, by using \eqref{Nehari}, we see that
		$\Vert(u, v) \Vert^2 = \lambda(\Vert u \Vert^p_p + \Vert v \Vert^q_q)$
		which implies
		$\lambda (2 - p)\Vert u \Vert^p_p + \lambda(2 - q)\Vert v \Vert^q_q < 0.$
		This leads to a contradiction proving that $\mathcal{N}^-_{\lambda} \subset \mathcal{A}$.
	\end{rmk}
	
	Now, we shall consider the functionals $R_n, R_e: \mathcal{A} : \to \mathbb{R}$ associated with the parameter $\lambda > 0$ in the following form
	\begin{equation}\label{R_n} 
		R_n(u,v)=\frac{\Vert(u,v)\Vert ^2 -\theta  \int_\mathbb{R^N} |u|^\alpha |v|^\beta dx}{ \Vert u \Vert _{p} ^{p} + \Vert v \Vert _{q} ^{q}}
	\end{equation}
	and
	\begin{equation}\label{R_e1} 
		R_e(u, v) =\frac{\frac{1}{2}\Vert(u,v)\Vert ^2-\frac{\theta}{\alpha+\beta}\int_{\mathbb{R}^N} \vert u\vert^\alpha \vert v \vert^\beta dx}{ \frac{1}{p}{\Vert u \Vert}_{p}^{p}+\frac{1}{q}{\Vert v \Vert}_{q}^{q}}
	\end{equation}
	Furthermore, we shall consider the following extremals
	\begin{equation}\label{Lambda_n } 
		\Lambda_n(u, v) = \sup_{t>0} R_n(tu, tv), \ \ \ \ \ \ \lambda^* = \inf_{(u, v) \in \mathcal{A}} \Lambda_n(u, v), 
	\end{equation}
	\begin{equation}\label{Lambda_e } 
		\Lambda_e (u, v) = \sup_{t>0} R_e(tu, tv), \ \ \ \ \ \ \lambda_* = \inf_{(u, v) \in \mathcal{A}} \Lambda_e(u, v).
	\end{equation}
	It is important to stress that $R_n$, $R_e$ belongs to $C^1(\mathcal{A}, \mathbb{R})$. Similarly, we observe that $\Lambda_n$ and $\Lambda_e$ are in $C^1(\mathcal{A}, \mathbb{R})$.	
	Under our assumptions, the energy functional $E_\lambda$ is bounded from below in $\mathcal{N}_{\lambda}$. Hence, we can consider the minimization problems given by \eqref{CN-} and \eqref{CN+}. 
	
	Our main objective in the present work is to find the largest positive value $\lambda^*> 0$ such that for each $0 < \lambda < \lambda^*$ the set $\mathcal{N}^0_{\lambda^*}$ is empty. This feature enables us to employ the Lagrange Multiplier Theorem. Our approach relies on an investigation of the existence and multiplicity of solutions for the System \eqref{sistema Principal} using the Nehari method together with the nonlinear Rayleigh quotient method. Under these conditions, we can state our main first result as follows:
	\begin{theorem}\label{teorem1}
		Suppose (P), ($V_0$) and ($V_1$). Then, $0 < \lambda_* < \lambda^* <\infty$. Furthermore, for each $\lambda \in (0, \lambda^*)$ the System (\ref{sistema Principal}) admits at least a weak solution $(u, v) \in \mathcal{A}$ for each $\theta > 0$. Furthermore, $(u,v)$ satisfies the following statements:
		\begin{itemize}
			\item[i)] $E''_{_\lambda}(u, v)(u, v)^2 > 0$, that is, $(u, v) \in \mathcal{N}^+_\lambda  \cap \mathcal{A}$; 
			\item[ii)] There exists $C > 0$ such that $E_{_\lambda}(u, v) \leq - C$.
		\end{itemize}
	\end{theorem}	
	Similarly, we can also prove the following main result:
	\begin{theorem}\label{teorem2}
		Suppose (P), ($V_0$) and ($V_1$). Then,  for each $\lambda \in (0, \lambda^*)$, the System \eqref{sistema Principal} admits at least a weak solution $(z, w) \in \mathcal{A}$ for each $\theta > 0$. Moreover, $(u,v)$ satisfies the following assertions: 
		\begin{itemize}
			\item[i)] $E''_{_\lambda}(z, w)(z, w)^2 < 0$, that is, $(z, w) \in \mathcal{N}^-_\lambda$;
			\item[ii)] For each $\lambda \in (0, \lambda_*)$ we obtain that $E_{_\lambda}(z, w) > 0$;
			\item[iii)] For $\lambda = \lambda_*$ it follows that $E_{_\lambda}(z, w) = 0$;
			\item[iv)] For each $\lambda \in (\lambda_*, \lambda^*)$ we obtain that $E_{_\lambda}(z, w) < 0$.
		\end{itemize}
	\end{theorem}
	
	\begin{theorem}\label{teorem3}
		Suppose (P), ($V_0$) and ($V_1$). Then the System \eqref{sistema Principal} has at least two weak solutions $(u, v)$ and $(z, w)$ for each $\lambda \in (0, \lambda^*)$ and for each $\theta > 0$. Furthermore, the functions $u$, $v$, $z$, and $w$ are strictly positive in $\mathbb{R}^N$.
	\end{theorem}
	It is worthwhile to mention that in our main result we do not need any restrictions on the size of the parameter $\theta > 0$. More precisely, we can prove that the solutions $(u,v)$ and $(z,w)$ given in Theorem \ref{teorem1} and Theorem \ref{teorem2} are not semitrivial for each $\theta > 0$. Recall that $(u, 0)$ and $(0,v)$ are said to be semitrivial solutions for the System \eqref{sistema Principal} whenever $u$ and $v$ are respectively weak solutions for the following scalar elliptic problems:
	\begin{equation}
		(-\Delta)^su +V_1(x)u =   \lambda|u|^{p - 2}u,  \;\;\; \mbox{in}\;\;\; \mathbb{R}^N, u\in H^s(\mathbb{R}^N),
	\end{equation}
	\begin{equation}
		(-\Delta)^s v +V_2(x)v =   \lambda|v|^{q - 2}v,  \;\;\; \mbox{in}\;\;\; \mathbb{R}^N, v\in H^s(\mathbb{R}^N).
	\end{equation}
	\subsection{Notation} Throughout this work we shall use the following notation:
	\begin{itemize}
		\item $E''_{_\lambda}(u, v)((u, v)(u, v)) = E''_{_\lambda}(u, v)(u, v)^2$ denotes the second derivatives in the $(u,v)$ direction.
		\item The norm in $L^{r}(\mathbb{R}^N)$ and $L^{\infty}(\mathbb{R}^N)$, will be denoted respectively by $\|\cdot\|_{r}$ and $\|\cdot\|_{\infty},  r \in [1, \infty)$.
		\item $S_r$ denotes the best constant for the embedding $X\hookrightarrow L^r(\mathbb{R}^N)$ for each $r \in [2, 2_s^*]$.
		\item $B_\epsilon = B_\epsilon(u, v) = \{(w, z) \in X: \Vert (u, v) - (w, z) \Vert < \epsilon\}$.
		\item $B_\delta(r) = \{x \in \mathbb{R}^N: \vert x - r \vert < \delta \}$.
	\end{itemize}
	
	\subsection{Outline} The remainder of this work is organized as follows: In the forthcoming section we consider some results concerning on the Nehari method for our main problem. In Section $3$ is devoted to the asymptotically behavior of solutions obtained in the Nehari manifolds $\mathcal{N}_{\lambda}^+$ and $\mathcal{N}_{\lambda}^-$. In Section 4 is proved our main results looking for the energy levels for each minimizer in the Nehari manifolds $\mathcal{N}_{\lambda}^+$ and $\mathcal{N}_{\lambda}^-$. In an Appendix we consider some further results for nonlocal elliptic problems involving technical estimates which are useful for the nonlinear Rayleigh quotient taking into account the energy functional $E_\lambda$.

	\section{Preliminary results and variational setting}
	As stated in the introduction, the main objective here is to ensure exists at least two positive solutions for our main problem using the Nehari Method together with the nonlinear Rayleigh quotient. Firstly, we consider the fibering maps for the energy functional. At the same time, we shall consider the fibering map for the associated nonlinear Rayleigh quotients. This relationship will be fundamental to finding critical points for the energy functional $E_{\lambda}$. Firstly, we consider the following result:
	
	\begin{rmk}\label{obs rel Rn eE'}
		Let $t > 0$ and $(u, v) \in \mathcal{A}$ be fixed. Hence, by using (\ref{R_n}), we obtain that 
		\begin{itemize} 
			\item[i)] $R_n(tu, tv) = \lambda$ if, and only if, $E'_{_\lambda}(tu, tv)(tu, tv) = 0$,
			\item [ii)] $R_n(tu, tv) > \lambda$  if, and only if, $E'_{_\lambda}(tu, tv)(tu, tv) > 0$,
			\item [iii)] $R_n(tu, tv) < \lambda$ if, and only if, $E'_{_\lambda}(tu, tv)(tu, tv) < 0$.
		\end{itemize}
	\end{rmk}
	
	\begin{rmk}\label{obs rel Re e E}
		Let $t > 0$ and $(u, v) \in \mathcal{A}$ be fixed. Theferore, by using (\ref{R_e1}), we infer that
		\begin{itemize}
			\item [i)]$ R_e(u, v) = \lambda$  if, and only if,  $E_{_\lambda}(u, v) = 0,$
			\item [ii)]$ R_e(u, v) > \lambda$  if, and only if, $E_{_\lambda}(u, v) > 0,$
			\item [iii)] $ R_e(u, v) < \lambda$  if, and only if, $E_{_\lambda}(u, v) < 0.$
		\end{itemize}
	\end{rmk}
	Now, we consider the fibering function $Q_n : \mathbb{R} \to \mathbb{R}$, for each $t > 0$, which is given by
	\begin{equation}\label{Q_n} 
		Q_n(t) = R_n(tu,tv) = \frac{t^2 \Vert(u,v)\Vert ^2 -\theta t^{\alpha + \beta}  \int_\mathbb{R^N} |u|^\alpha |v|^\beta dx}{ t^{p}\Vert u \Vert _{p} ^{p} + t^{q}\Vert v \Vert _{q} ^{q}},
	\end{equation}
	Analogously, we consider $Q_e : \mathbb{R} \to \mathbb{R}$, for each $t > 0$, given by 
	\begin{equation}\label{Q_e} 
		Q_e(t) =  R_e(tu, tv) = \frac{\frac{1}{2}t^2\Vert(u,v)\Vert ^2-\frac{\theta }{\alpha+\beta}t^{\alpha +\beta}\int_{\mathbb{R}^N} \vert u\vert^\alpha \vert v \vert^\beta dx}{ \frac{t^{p}}{p}{\Vert u \Vert}_{p}^{p}+\frac{t^{q}}{q}{\Vert v \Vert}_{q}^{q}}.
	\end{equation}
	It is not hard to verify that for small $t > 0$ we obtain that
	$$
	\lim_{t \rightarrow 0} \frac{Q_n(t)}{t^{2 - p}} >\frac{ \Vert (u, v) \Vert ^2}{\Vert u \Vert _{p}^{p}+\Vert v \Vert _{q}^{q}} > 0,
	\;\;\;\;\lim_{t \rightarrow 0}\frac{Q_n'(t)}{t^{1- p}} > \frac{\Vert u \Vert _{p} ^{p}\Vert (u,v)\Vert  ^{2} }{(\Vert u \Vert _{p} ^{p} + \Vert v \Vert _{q} ^{q})^2} >0,
	$$
	
	$$\lim_{t \to 0}\frac{Q_e(t)}{t^{2 - p}}  >  \frac{\frac{1}{2}\Vert(u,v)\Vert ^2}{\frac{1}{p}{\Vert u \Vert}_{p}^{p}+\frac{1}{q}{\Vert v \Vert}_{q}^{q}}
	>  0, \,\, \,\, \lim_{t \to 0} \frac{Q_e'(t)}{t^{1 - p}} >  \left ( \frac{1}{p} - \frac{1}{2}  \right )\displaystyle \frac{\frac{1}{2}\Vert u \Vert _{p}^{p} \Vert(u, v) \Vert ^2}{ \left ( \frac{1}{p} \Vert u \Vert_{p}^{p} + \frac{1}{q} \Vert v \Vert_{q}^{q} \right )^2} > 0.$$
	Similarly, we infer that  
	
	$$ \lim_{ t \rightarrow \infty}\frac{Q_n(t)}{t^{\alpha +\beta - p}}<  \frac{-\theta \int _\mathbb{R^N} |u|^{\alpha}  |u|^{\beta}dxdy}{\Vert u \Vert_{p} ^{p} + \Vert v \Vert_{q} ^{q}}<0$$
	and
	\begin{eqnarray*}
		\lim_{ t \rightarrow \infty}  \frac{Q_n'(t)}{t^{\alpha +\beta- 1 - p}}& < &-\frac{ \theta (\alpha + \beta - 1) \Vert u \Vert_{p} ^{p}\int _\mathbb{R^N} |u| ^{\alpha}  |v|^{\beta}dx }{(\Vert u \Vert_{p} ^{p} + \Vert v \Vert_{q} ^{q})^2} - \frac{ \theta (\alpha + \beta - 1) t^{q- p}\Vert v \Vert_{q} ^{q}\int _\mathbb{R^N} |u|^{\alpha}  |v|^{\beta}dx}{{(\Vert u \Vert_{p} ^{p} + \Vert v \Vert_{q} ^{q})^2}} < 0.
	\end{eqnarray*}
	At the same time, we observe that 
	$$\lim_{t \to \infty}\frac{Q_e(t)}{t^{\alpha +\beta - p}} <  \frac{-\frac{\theta}{\alpha+\beta}\int_{\mathbb{R}^N} |u|^\alpha|v|^\beta dx}{\frac{1}{p}{\Vert u \Vert}_{p}^{p}+\frac{1}{q}{\Vert v \Vert}_{q}^{q}} <  0$$
	and 
	\begin{eqnarray*}
		\lim_{t \to \infty} \frac{Q_e'(t)}{t^{\alpha +\beta - 1 - p}}& < &\left ( \frac{\theta}{\alpha + \beta} - \frac{\theta}{p} \right )  \Vert u \Vert _{p}^{p} \int_{\mathbb{R}^N} {|u|^\alpha |v|  ^\beta dx} +  \left ( \frac{\theta}{\alpha + \beta} - \frac{\theta}{q} \right )  \Vert v \Vert _{q}^{q} \int_{\mathbb{R}^N} {|u|^\alpha |v|^\beta dx} <  0.
	\end{eqnarray*}
	
	It important to find the solutions of $Q_n'(t) = 0$ with $t > 0$. Now, by using the Implicit Function Theorem and Lemma \ref{apendice Maxwell} given in the Appendix, there exists a unique $t =t_n(u, v)$ such that $Q_n'(t) = 0$ where $t_n: \mathcal{A} \to \mathbb{R}$ belongs to $C^1(\mathcal{A}, \mathbb{R})$.
	Now, we consider 
	$\Lambda _n, \Lambda _e: \mathcal{A} \to \mathbb{R}$ defined by 
	\begin{equation}\label{Lambda_n de tn} 
		\Lambda_n(u, v) = R_n(t_n(u, v) (u,v)),
	\end{equation}
	\begin{equation}\label{Lambda_e de te} 
		\Lambda_e(u, v) = R_e(t_e(u, v) (u,v)),
	\end{equation}
	It is not hard to verify that $\Lambda_n, \Lambda_e \in C^1(\mathcal{A}, \mathbb{R})$.
	
	\begin{rmk}\label{tn p = q}
		Assuming that $p = q$ the solution of $Q_n'(t) = 0$ is given explicitly in the following form:
		\begin{equation}\label{t__n explic} 
			t = t_n(u, v) = \left ( \frac{(2 - q)\Vert (u, v) \Vert ^2}{\theta (\alpha + \beta - q)  \int_\mathbb{R^N} \vert u\vert^\alpha \vert v\vert^\beta dx} \right)^\frac{1}{\alpha + \beta -2}.
		\end{equation}
		Once again the functional $\Lambda _n: \mathcal{A} \to \mathbb{R}$ defined by
		$$\Lambda _n(u, v) := \max\limits_{t>0}R_n(tu, tv) = R_n(t_n(u, v) (u,v)).$$
		Under these conditions, assuming that $p = q$,  we obtain
		\begin{equation}\label{Lambda_n} 
			\Lambda _n(u, v) = C_{\alpha, \beta, q, \theta} \frac{\left (\Vert (u, v) \Vert ^2\right )^{\frac{ \alpha +\beta -q}{ \alpha + \beta -2}} \left ( \int_\mathbb{R^N} |u|^{\alpha} |v|^{\beta}dx\right )^{\frac{q-2}{\alpha +\beta -2}}} {\Vert u \Vert _{q} ^{q} + \Vert v \Vert _{q}^{q}},
		\end{equation}
		where
		$$C_{\alpha, \beta, q, \theta} = \theta ^{\frac{q-2}{\alpha +\beta - 2}} ( \alpha + \beta - 2)(\alpha+\beta - q)^\frac{-(\alpha +\beta -q)}{\alpha +\beta - 2}(2 - q)^\frac{2 - q}{\alpha + \beta - 2}.$$
	\end{rmk}
	
	\begin{rmk}\label{te p = q}
		Assuming that $p = q$ the solution of $Q_e'(t) = 0$ is given explicitly in the following form:
		$$ t = t_e(u,v) = \left ( \frac{\left ( 2 - q \right )\left ( \alpha + \beta \right ) \Vert (u, v) \Vert ^2}{ 2\theta\left (\alpha +\beta - q \right ) \int _{\mathbb{R}^N} {\vert u \vert^\alpha \vert v\vert^\beta dx}} \right )^ \frac{1}{\alpha +\beta -2}.
		$$
		As a consequence, we also obtain that 
		\begin{equation}\label{Lambda_e = R_e (te(u, v)) q } 
			\Lambda _e(u, v) = \tilde{C}_{\alpha, \beta, q, \theta} \frac{\left (\Vert (u, v) \Vert ^2\right )^{\frac{\alpha +\beta -q}{ \alpha + \beta -2}} \left ( \int_\mathbb{R^N} |u|^{\alpha} |v|^{\beta}dx\right )^{\frac{q-2}{\alpha +\beta -2}}} {\Vert u \Vert _{q} ^{q} + \Vert v \Vert _{q}^{q}},
		\end{equation}
		and
		\begin{equation}\label{Ctilalphabetaqteta} 
			\tilde{C}_{\alpha, \beta, q, \theta} = q (\alpha + \beta - 2) \left ( \frac{(2 - q)(\alpha + \beta)}{\theta} \right )^{\frac{2 - q}{\alpha + \beta - 2}} \left ( \frac{1}{2(\alpha + \beta - q)} \right )^{\frac{\alpha + \beta - q}{\alpha + \beta - 2}}.
		\end{equation}
	\end{rmk}
	
	\begin{rmk}
		Suppose (P), ($V_0$) and ($V_1$).  Assume also that $p = q$. Then $C_{\alpha, \beta, q, \theta} > \tilde{C}_{\alpha, \beta, q,\theta}$, see for instance  Lemma \ref{apendice2A} in the Appendix.
	\end{rmk}
	
	One of the main features in the present work is to consider the cases $p = q$ and $p \neq q$. Recall that for the case $p \neq q$ all functionals $\Lambda_n, \Lambda_e$ and $t_n, t_e$ are given only implicitly. Hence, we can prove our main results using the Nehari method and the Nonlinear Rayleigh quotient assuming that those functionals are given only implicitly. Firstly, we consider the following result:
	
	\begin{prop}\label{lema_naru1}
		Suppose (P), ($V_0$) and ($V_1$). Consider the functional $\Lambda_n: \mathcal{A} \to \mathbb{R}$ given by
		$$\Lambda_n(u, v) = Q_n(t_n(u, v)) = \max_{t>0}R_n(tu, tv) = R_n(t_n(u, v) (u,v)),$$ $(u, v) \in \mathcal{A}$. Then we obtain that $\Lambda_n$ is 0-homogeneous.
	\end{prop}
	\begin{proof}
		For each $s > 0$, $s \in \mathbb{R}$,
		$\Lambda _n(su, sv) = \sup\limits_{t>0}Q_n(tsu, tsv) = \sup\limits_{a>0} Q_n(au, av) = \Lambda _n(u, v).$ This ends the proof.
	\end{proof}
	
	\begin{prop}\label{lema_naru2}
		Suppose (P), ($V_0$) and ($V_1$). Let $\Lambda_e: \mathcal{A} \to \mathbb{R}$ given in
		$$\Lambda_e(u, v) = Q_e(t_e(u, v)) = \max_{t>0}R_e(tu, tv) = R_e(t_e(u, v) (u,v)),$$ $(u, v) \in \mathcal{A}$. Then we obtain that $\Lambda_e$ is 0-homogeneous.
	\end{prop}
	\begin{proof}
		For each $s > 0$, $s \in \mathbb{R}$, $\Lambda _e(su, sv) = sup_{t>0}Q_e(tsu, tsv) = sup_{a>0} Q_e(au, av) = \Lambda _e(u, v)$. This ends the proof.
	\end{proof}
	
	\begin{prop}\label{lema_naruprop}
		Suppose (P), ($V_0$) and ($V_1$). Then we obtain that
		$$
		Q_n(t) - Q_e(t)=\frac{t}{pq}\left ( \frac{qt^{p} \parallel u \parallel _{p}^{p} + p t^{q} \parallel v \parallel _{q}^{q}}{t^{p} \parallel u \parallel _{p}^{p} + t^{q} \parallel v \parallel _{q}^{q}} \right )\frac{dQ_e(t)}{dt}.
		$$
	\end{prop}
	\begin{proof}
		Firstly, by using (\ref{Q_n}) and (\ref{Q_e}), we infer that 
		\begin{eqnarray*}
			Q_n(t) - Q_e(t) & = & \frac{\left (t^2 \Vert(u,v)\Vert ^2 -\theta t^{\alpha + \beta} \int\limits_\mathbb{R^N} |u|^\alpha |v|^\beta dx\right )\left ( \frac{1}{p}{t^{p}\Vert u \Vert}_{p}^{p}+\frac{1}{q}{t^{q}\Vert v \Vert}_{q}^{q} \right )}{\left (t^{p}\Vert u \Vert _{p} ^{p} + t^{q}\Vert v \Vert _{q} ^{q}\right )\left (\frac{1}{p}{t^{p}\Vert u \Vert}_{p}^{p}+\frac{1}{q}{t^{q}\Vert v \Vert}_{q}^{q} \right )} \\ 
			& & - \frac{\left (\frac{1}{2}t^2\Vert(u,v)\Vert ^2-\frac{\theta}{\alpha+\beta}t^{\alpha+\beta}\int\limits_{\mathbb{R}^N} |u|^\alpha|v|^\beta dx \right )\left (t^{p}\Vert u \Vert _{p} ^{p} + t^{q}\Vert v \Vert _{q} ^{q} \right )}{\left (t^{p}\Vert u \Vert _{p} ^{p} + t^{q}\Vert v \Vert _{q} ^{q}\right )\left (\frac{1}{p}{t^{p}\Vert u \Vert}_{p}^{p}+\frac{1}{q}{t^{q}\Vert v \Vert}_{q}^{q}\right )}.
		\end{eqnarray*}
		Hence, we obtain that 
		$$Q_n(t) - Q_e(t)=\frac{t}{pq}\left ( \frac{qt^{p} \Vert u \Vert _{p}^{p} + p t^{q} \Vert v \Vert _{q}^{q}}{t^{p} \Vert u \Vert _{p}^{p} + t^{q} \Vert v \Vert _{q}^{q}} \right )\frac{d}{dt}Q_e(t).$$
		This ends the proof. 
	\end{proof}
	As a consequence, we obtain the following result:
	\begin{rmk}\label{lema_naru}
		Suppose (P), ($V_0$) and ($V_1$). Then, we obtain that following assertions: 
		\begin{itemize}
			\item[i)] There holds $Q_n(t) > Q_e(t) $ if and only if $\frac{dQ_e(t)}{dt} > 0$. Moreover, the last inequality occurs if and only if $t < t_e(u, v)$;
			\item[ii)]	It holds $Q_n(t) < Q_e(t)$ if and only if $\frac{dQ_e(t)}{dt} < 0$. Similarly, the last inequality holds if and only if  $t > t_e(u, v)$;
			\item[iii)]	It holds $Q_n(t) = Q_e(t)$ if and only if $\frac{dQ_e(t)}{dt} = 0$. Once again the last identity holds if and only if $t = t_e(u, v).$
		\end{itemize}
	\end{rmk}
	
	Now, we shall consider the following assertions around the functionals $\Lambda_n, \Lambda_e: \mathcal{A} \to \mathbb{R}$.
	\begin{rmk}\label{integral limitada tilde u}
		Now, by using (\ref{Lambda_n de tn}) together with the identities $Q_n(t) = R_n(tu, tv)$ and $\left.\frac{d}{dt}R_n(t \bar{u}_k, t\bar{{v}_k})\right|_{t = 1}  = 0$, we also obtain that
		\begin{equation}\begin{array}{lll}\label{integra acoplamento}
				\theta \displaystyle \int\limits_ \mathbb{R^N} \vert \bar{u}_k \vert ^\alpha \vert \bar{v}_k \vert ^\beta dx = \frac{\displaystyle  \Vert (\bar{u}_k, \bar{v}_k ) \Vert ^2\left [ \left (2 - p \right ) \Vert \bar{u}_k \Vert_{p}^{p} + \left (2 - q \right ) \Vert \bar{v}_k \Vert_{q}^{q} \right ]}{\left [ \left ((\alpha + \beta \right ) - p) \Vert \bar{u}_k \Vert_{p}^{p} + \left ((\alpha + \beta \right ) - q) \Vert \bar{v}_k \Vert_{q}^{q}\right ]}.
			\end{array}
		\end{equation}
		In light of Lemma \ref{ex3.18Elon analise reta} given in the Appendix we infer that 
		\begin{equation}\begin{array}{lll}\label{2k}
				f(q) \Vert (\tilde{u}_k, \tilde{v}_k) \Vert^2 \le\theta  \displaystyle\int\limits_ \mathbb{R^N} \vert \tilde{u}_k \vert ^\alpha \vert \tilde{v}_k \vert ^\beta dx \le   f(p) \Vert (\tilde{u}_k,\tilde{v}_k ) \Vert ^2.
			\end{array}
		\end{equation}
	\end{rmk}

	\begin{rmk}\label{integral limitada bar u}
		Similarly, using the equation $Q_e(t) = R_e(t \bar{u}, t \bar{v})$ and the fact that $\left.\frac{d}{dt}R_e(t \bar{u}_k, t\bar{{v}_k})\right|_{t = 1} = 0$, we also obtain that
		\begin{equation}\begin{array}{lll}\label{integra acoplamento bar}
				\theta \displaystyle \int\limits_ \mathbb{R^N} \vert \bar{u}_k \vert ^\alpha \vert \bar{v}_k \vert ^\beta dx = \frac{  \Vert (\bar{u}_k, \bar{v}_k ) \Vert ^2\left [ \left ( \frac{1}{p} -  \frac{1}{2} \right ) \Vert \bar{u}_k \Vert_{p}^{p} + \left ( \frac{1}{q} -  \frac{1}{2} \right )  \Vert \bar{v}_k \Vert_{q}^{q} \right ]}{\left [ \left (\frac{1}{p} - \frac{1}{\alpha + \beta} \right ) \Vert \bar{u}_k \Vert_{p}^{p} + \left ( \frac{1}{q} - \frac{1}{\alpha + \beta} \right ) \Vert \bar{v}_k \Vert_{q}^{q}\right ]}.
			\end{array}
		\end{equation}
		Furthermore, by using the Lemma \ref{ex3.18Elon analise reta} in the  Appendix, we deduce that
		\begin{equation}\begin{array}{lll}\label{2ky}
				\bar{f}(q) \Vert (\bar{u}_k, \bar{v}_k) \Vert^2 \le\theta \int\limits_ \mathbb{R^N} \vert \bar{u}_k \vert ^\alpha \vert \bar{v}_k \vert ^\beta dx \le   \bar{f}(p) \Vert (\bar{u}_k, \bar{v}_k ) \Vert ^2.
			\end{array}
		\end{equation}
	\end{rmk}
	
	\begin{prop}\label{Limitada}
		Suppose (P), ($V_0$) and ($V_1$). Let $(u_k, v_k) \in \mathcal{A}$ be a minimizer sequence for $\lambda^*$. Then there exists another bounded minimizing sequence for $\lambda^*$ in $X$.
	\end{prop}
	\begin{proof}    
		Define the sequence $(\bar{u}_k, \bar{v}_k)$ given by $\bar{u}_k = t_n(u_k, v_k) u_k$ and $\bar{v}_k = t_n(u_k, v_k)v_k$.
		Since $\Lambda_n$ is zero homogeneous it follows that $t_n(\bar{u}_k, \bar{v}_k) = 1$. In particular,  $ \lambda^* \leq \Lambda_n(\bar{u}_k, \bar{v}_k) = R_n(\bar{u}_k, \bar{v}_k) \leq \lambda^* + 1$ and $R'_n(\bar{u}_k, \bar{v}_k)(\bar{u}_k, \bar{v}_k) = 0$. Hence, using the last assertion together with Remark \ref{integral limitada tilde u}, we obtain that $(\bar{u}_k, \bar{v}_k)$ is bounded in $X$. This ends the proof. 
	\end{proof}
	\begin{prop}\label{Limitada2}
		Suppose (P), ($V_0$) and ($V_1$). Let $(u_k, v_k) \in \mathcal{A}$ be a minimizer sequence for $\lambda_*$. Then there exists another bounded minimizing sequence for $\lambda_*$ in $X$.
	\end{prop}
	\begin{proof}
		Consider the sequence $(\bar{u}_k, \bar{v}_k)$ given by $\bar{u}_k = t_e(u_k, v_k) u_k$ and $\bar{v}_k = t_e(u_k, v_k)v_k$.
		Once again, by using the fact that $\Lambda_e$ is zero homogeneous and Remark \ref{integral limitada bar u}, we infer that $(\bar{u}_k, \bar{v}_k)$ is bounded in $X$. This ends the proof. 
	\end{proof}

	\begin{prop}\label{lambda frac Sem Inf}
		Suppose (P), ($V_0$) and ($V_1$). Let $(u_k, v_k) \in \mathcal{A}$ be a minimizer sequence for $\Lambda_n$ and $\bar{u}_k = t_n(u_k, v_k) u_k, \bar{v}_k = t_n(u_k, v_k)v_k$.
		Then we obtain that
		$$\Lambda_n(\bar{u},\bar{v}) \le \lim \inf_{k \to \infty}\Lambda_n(\bar{u}_k,\bar{v}_k)$$
		where $\bar{u}_k \rightharpoonup \bar{u}$ in $X_1$ and $ \bar{v}_k \rightharpoonup \bar{v}$ in $X_2$.
	\end{prop}
	\begin{proof}
		Firstly, the sequence $(\bar{u}_k,\bar{v}_k)$ is bounded in $X$, see Proposition \ref{Limitada}. Hence, there exists $(\bar{u},\bar{v})$ such that $\bar{u}_k \rightharpoonup \bar{u}$ in $X_1$ and $ \bar{v}_k \rightharpoonup \bar{v}$ in $X_2$. Recall that $(u,v) \mapsto \Vert ( u, v ) \Vert^2$ is weakly lower semicontinuous. Furthermore, using Proposition \ref{chave} and Dominated Convergence Theorem, we infer that 
		$$\lim\limits_{k \to \infty}\int_\mathbb{R^N} |\bar{u}_k|^{\alpha} |\bar{v}_k|^{\beta}dx = \int_\mathbb{R^N} |\bar{u}|^{\alpha} |\bar{v}|^{\beta}dx$$
		and 
		$$\lim\limits_{k \to \infty}\Vert \bar{u}_k \Vert _{p} ^{p} = \Vert \bar{u} \Vert _{p} ^{p}, \lim\limits_{k \to \infty}\Vert \bar{v}_k \Vert _{q} ^{q} = \Vert \bar{v} \Vert _{q} ^{q}.$$
		Therefore, the functional $R_n: X \setminus \{0\} \to \mathbb{R}$ satisfies $$R_n(\bar{u},\bar{v}) \le \lim \inf_{k \to \infty} R_n(\bar{u}_k, \bar{v}_k).$$
		Furthermore, by using the fact that $\Lambda_n$ is zero homogeneous, we mention that $R_n(\bar{u}_k, \bar{v}_k) = \Lambda_n(\bar{u}_k, \bar{v}_k) = \Lambda_n(u_k, v_k)$. In particular, by using the fact that $R_n(t u, tv) \leq \Lambda_n(u,v), t \geq 0, (u,v) \in \mathcal{A}$, we obtain that
		\begin{equation*}
			\Lambda_n(\bar{u},\bar{v}) = R_n(t_n(\bar{u},\bar{v})(\bar{u},\bar{v})) \leq	\liminf_{k \to \infty} R_n(t_n(\bar{u},\bar{v})(\bar{u}_k,\bar{v}_k)) \le \liminf_{k \to \infty} R_n(\bar{u}_k, \bar{v}_k) = \liminf_{k \to \infty} \Lambda_n(\bar{u}_k, \bar{v}_k).
		\end{equation*}
		This finishes the proof. 
	\end{proof}
	
	Now, we observe also that the functional $E_\lambda: X \to \mathbb{R}$ is weakly lower semicontinuous. Hence, $E_{_\lambda} (u, v) \le  \liminf_{k \to \infty} E_{_\lambda} (u_k, v_k)$. Similarly, we also obtain that $E'_{_\lambda}(u, v)(u, v) \le  \liminf_{k \to \infty} E'_{_\lambda}(u_k, v_k)(u_k, v_k)$ and $E''_{_\lambda}(u, v)(u, v)^2 \le  \liminf_{k \to \infty} E''_{_\lambda}(u_k, v_k)(u_k, v_k)^2$ where $(u_k,v_k) \rightharpoonup (u,v)$ in $X$ and $(u,v) \in X$.
	\begin{prop}\label{lambda e frac Sem Inf}
		Suppose $(P)$, ($V_0$) and ($V_1$). Let $(u_k, v_k) \in \mathcal{A}$ be a minimizer sequence for $\lambda_*$. Consider the sequence $\bar{u}_k = t_e(u_k, v_k) u_k,	\bar{v}_k = t_e(u_k, v_k)v_k$.
		Then, we obtain that
		$$\Lambda_e(\bar{u},\bar{v}) \le \liminf_{k \to \infty}\Lambda_e(\bar{u}_k,\bar{v}_k)$$
		where $\bar{u}_k \rightharpoonup \bar{u}$ in $X_1$ and $\bar{v}_k \rightharpoonup \bar{v}$ in $X_2$.
	\end{prop}
	\begin{proof}
		The proof follows the same lines discussed in the proof of Proposition \ref{lambda frac Sem Inf}. We omit the details. 
	\end{proof}
	
	\begin{prop}\label{E coersiva}
		Suppose $(P)$, ($V_0$) and ($V_1$). Then the energy functional given in (\ref{Ener}) is coercive in Nehari set $\mathcal{N}_\lambda$.
	\end{prop}
	\begin{proof}    
		Let $(u, v) \in \mathcal{N}_{\lambda}$ be a fixed function.  It follows from $E'_{_\lambda}(u, v)(u, v) = 0$ that 
		$$\Vert (u, v) \Vert ^2 - \lambda\Vert u\Vert^{p}_{p} - \lambda\Vert v \Vert^{q}_{q} = \theta\int\limits_{\mathbb{R}^N} |u|^\alpha|v|^\beta dx.$$ 
		As a consequence, we obtain that $$E_{_\lambda}(u, v) = \left ( \frac{1}{2} - \frac{1}{\alpha + \beta} \right )\Vert (u, v) \Vert ^2 + \lambda\left ( -\frac{1}{p} + \frac{1}{\alpha + \beta}\right )\Vert u \Vert^{p}_{p} + \lambda\left ( -\frac{1}{q} + \frac{1}{\alpha + \beta}\right )\Vert v \Vert^{q}_{q}.$$ 
		In light of Proposition \ref{chave} we infer that 
		$$E_{_\lambda}(u, v) \ge \left ( \frac{1}{2} - \frac{1}{\alpha + \beta} \right )\Vert (u, v) \Vert ^2 + \lambda\left ( -\frac{1}{p} + \frac{1}{\alpha + \beta}\right )S_{p}^{p}\Vert u \Vert^{p} + \lambda\left ( -\frac{1}{q} + \frac{1}{\alpha + \beta}\right )S_{q}^{q}\Vert v \Vert^{q}.$$
		Therefore, $$E_{_\lambda}(u, v) \ge C_1 \Vert (u, v) \Vert ^2 + C_2 {max \{\Vert(u, v) \Vert^{p}}, \Vert(u, v)\Vert^{q}\}, (u,v) \in \mathcal{N}_\lambda$$ where $C_1 > 0$ and $C_2 > 0$. In view of hypothesis $(P)$ we deduce that $E_{_\lambda}(u,v) \to +\infty$ as $\Vert (u, v) \Vert \to +\infty$ where $(u,v) \in \mathcal{N}_\lambda$. Hence, $E_{_\lambda}$ is coercive in the Nehari manifold $\mathcal{N}_{\lambda}$. This ends the proof.
	\end{proof}
	
	\begin{prop}\label{tilde Longe de zero}
		Suppose $(P)$, ($V_0$) and ($V_1$). Let $(u_k, v_k) \in \mathcal{A}$ be a minimizer sequence for $\lambda^*$. Then $(\bar{u}_k, \bar{v}_k)$, given in Proposition \eqref{Limitada} is bounded from below by a positive constant. More precisely, we obtain the following assertion: There exists $\delta > 0$ such that
		$\Vert (\bar{u}_k, \bar{v}_k)\Vert \ge \delta > 0, \ k \in \mathbb{N}.$
	\end{prop}
	\begin{proof}
		Firstly, by using (\ref{Lambda_n de tn}) together with $\frac{d}{dt}\Lambda_n(t \bar{u}_k, t\bar{{v}_k})\mid _{t=1} = 0$ and Lemma \ref{ex3.18Elon analise reta}, we infer that
		\begin{equation}\begin{array}{lll}\label{6t}
				\begin{array}{c}
					\Vert (\bar{u}_k, \bar{v}_k) \Vert ^2 \le \theta \displaystyle \frac{(\alpha +\beta - p)}{2 - p}  \int_{\mathbb{R}^N}\vert\bar{u}_k\vert ^{\alpha}\vert \bar{v}_k\vert ^{\beta}dx.
				\end{array}
			\end{array}
		\end{equation}
		Now, applying the inequality of H$\ddot{o}$lder and Lemma \ref{chave}, we see that 
		\begin{eqnarray*}
			\Vert(\bar{u}_k,\bar{v}_k)\Vert ^2&\le&\theta \frac{(\alpha +\beta - p)}{2 - p} \left (\int\limits_\mathbb{R^N} \vert \bar{u}_k\vert^{\alpha \frac{\alpha + \beta}{\alpha}}dx\right )^\frac{\alpha}{\alpha + \beta} \left (\int_\mathbb{R^N}\vert\bar{v}_k \vert^{\beta \frac{\alpha + \beta}{\beta}} dx\right )^{\frac{\beta}{\alpha + \beta}}\\ 
			&\le&\theta \frac{(\alpha +\beta - p)}{2 - p} \Vert \bar{u}_k\Vert_{\alpha + \beta}^{\alpha} \Vert\bar{v}_k \Vert_{\alpha + \beta}^{\beta} \\
			&\le&\theta S_{\alpha + \beta}^{\alpha + \beta} \frac{(\alpha +\beta - p)}{2 - p} \Vert \bar{u}_k\Vert^{\alpha} \Vert\bar{v}_k \Vert^{\beta}
			\le\theta S_{\alpha + \beta}^{\alpha + \beta} \frac{(\alpha +\beta - p)}{2 - p} \Vert (\bar{u}_k, \bar{v}_k)\Vert^{\alpha + \beta}.\\
		\end{eqnarray*}
		As a consequence, we obtain that
		\begin{equation*}
			\Vert(\bar{u}_k,\bar{v}_k)\Vert \geq \delta : =\left (  \frac{2 - p}{(\alpha +\beta - p)\theta S_{\alpha + \beta}^{\alpha + \beta}}\right )^{\frac{1}{\alpha + \beta - 2}}.
		\end{equation*}
		This ends the proof.
	\end{proof}
	
	\begin{prop}\label{barLonge de zero}
		Suppose $(P)$, ($V_0$) and ($V_1$). Let $(u_k, v_k) \in \mathcal{A}$ be a minimizer sequence for $\lambda_*$. Then the sequence $(\bar{u}_k, \bar{v}_k)$ given by Proposition \ref{Limitada2} is bounded from below by a positive constant, i.e, we obtain the following assertion: There exists $\delta > 0$ such that $\Vert (\bar{u}_k, \bar{v}_k)\Vert \ge \bar{\delta} > 0$ holds for each $k \in \mathbb{N}$.
	\end{prop}
	\begin{proof}
		Initially, using the same ideas discussed in the proof Proposition \ref{tilde Longe de zero} and taking into account that  $\frac{d}{dt}R_e(t \bar{u}_k, t\bar{{v}_k})\mid _{t=1}=0$, we obtain
		\begin{eqnarray*}
			\Vert(\bar{u}_k,\bar{v}_k)\Vert &\ge&\left (  \frac{(\alpha + \beta)f(q)}{2 \theta S_{\alpha + \beta}^{\alpha + \beta}}\right )^{\frac{1}{\alpha + \beta - 2}} = \bar{\delta} > 0.
		\end{eqnarray*}
		Notice also that the function $f:(1,2) \to \mathbb{R}$ is defined by  $f(x) = \frac{2 - x}{\alpha + \beta - x}$, see Lemma \ref{ex3.18Elon analise reta} in the Appendix.
		This finishes the proof.
	\end{proof}
	\begin{prop}\label{lambda*ating}
		Suppose $(P)$, ($V_0$) and ($V_1$). Then there exists $(u, v)\in \mathcal{A}$  such that  $$\lambda^* = \inf_{(z, w) \in \mathcal{A}}\Lambda_n(z,w) = \Lambda_n(u,v).$$ 
		Hence, the number $\lambda^* > 0$ is attained.
	\end{prop}
	\begin{proof}
		According to Proposition \ref{Limitada} we mention that $(\bar{u}_k, \bar{v}_k)$ is a bounded sequence in $X$. Hence, there exists $(\bar{u}, \bar{v}) \in X$ such that $(\bar{u}_k, \bar{v}_k)\rightharpoonup (\bar u, \bar v)$ in $X$. Moreover, by using \eqref{6t} and Proposition \ref{barLonge de zero}, we deduce that $(\bar u, \bar v) \in \mathcal{A}$. Under these conditions, by using Proposition \ref{lambda frac Sem Inf}, we obtain that
		$$\lambda^*\le \Lambda_n(\bar u, \bar v) \le \liminf_{k \to +\infty} \Lambda_n(\bar{u}_k, \bar{v}_k) = \lambda^*.$$
		In particular, we infer that $\lambda^* = \Lambda_n(\bar u, \bar v) = \inf_{(z, w) \in \mathcal{A}} \Lambda_n(z,w)$.
		Therefore, $\lambda^*$ is attained. This ends the proof.
	\end{proof}
	
	\begin{prop}\label{lambda_*ating}
		Suppose $(P)$, ($V_0$) and ($V_1$). Then there exists $(\bar{u}, \bar{v})\in \mathcal{A}$  such that  $$\lambda_* = \inf_{(z, w) \in \mathcal{A}}\Lambda_e(z,w) = \Lambda_e(\bar{u}, \bar{v}). $$ 
		In other words, the number $\lambda_* > 0$ is attained.
	\end{prop}
	\begin{proof}
		Firstly, by using the same ideas discussed in the proof of Proposition \ref{lambda*ating} and taking into account Proposition \ref{lambda e frac Sem Inf}, we infer that
		$$\lambda_* \le \Lambda_e(\bar u, \bar v) \le \liminf_{k \to +\infty} \Lambda_e(\bar{u}_k, \bar{v}_k) = \lambda_*,$$
		where $(\bar{u}_k, \bar{v}_k) \rightharpoonup (\bar u, \bar v)$ in $X$. Here we mention that the sequence $(\bar{u}_k, \bar{v}_k)$ is bounded in $X$, see Proposition \ref{Limitada2}. Furthermore, by using Proposition \ref{barLonge de zero}, we see that $(\bar u, \bar v) \in \mathcal{A}$. 
		Therefore, $\lambda_* = \Lambda_e(\bar u, \bar v) = \inf_{(z, w) \in \mathcal{A}}\Lambda_e(z,w)$. This ends the proof.
	\end{proof}
	
	\begin{prop}\label{esquema}
		Suppose $(P)$, ($V_0$) and ($V_1$). Then there exists $C_\delta > 0$ such that $\lambda^*\ge C_\delta > 0$. 
	\end{prop}
	\begin{proof}
		Consider the sequence $(\bar{u}_k, \bar{v}_k)$ given in Proposition \ref{Limitada}. It follows from (\ref{2k}) and (\ref{Lambda_n de tn}) that  
		$$\Lambda_n(\bar{u}_k, \bar{v}_k) \ge \frac{\left ( 1 - f(p) \right ) \Vert (\bar{u}_k, \bar{v}_k)\Vert^2}{\Vert \bar{u}_k\Vert^{p}_{p} + \Vert \bar {v}_k\Vert_{q}^{q}}.$$
		Here we mention that $f: (1, 2) \to \mathbb{R}$ is given by $f(x) =  (2 - x)/(\alpha +\beta - x)$. Using the Sobolev embedding given in Lemma \ref{chave}, we obtain that
		$$\Lambda_n(\bar{u}_k, \bar{v}_k) \ge \frac{\left ( 1 - f(p) \right )  \Vert (\bar{u}_k, \bar{v}_k)\Vert^2}{S_p^p\Vert \bar{u}_k\Vert^{p} + S_q^q\Vert \bar {v}_k\Vert^{q}}  \ge \frac{\left ( 1 - f(p) \right )  \Vert (\bar{u}_k, \bar{v}_k)\Vert^2}{S_p^p\Vert (\bar{u}_k, \bar{v}_k)\Vert^{p} + S_q^q\Vert (\bar{u}_k, \bar {v}_k)\Vert^{q}}.$$
		Define  $S = \max\{S_p^p, S_q^q\}$. Assuming $\Vert (\bar{u}_k, \bar{v}_k)\Vert >1$ the last assertion implies that
		$$\Lambda_n(\bar{u}_k, \bar{v}_k) \ge \frac{\left ( 1 - f(p) \right )  \Vert (\bar{u}_k, \bar{v}_k)\Vert^2}{2S\Vert (\bar{u}_k, \bar{v}_k)\Vert^{q}}.$$
		In view of Proposition \ref{tilde Longe de zero} we obtain that 
		$$\Lambda_n(\bar{u}_k, \bar{v}_k) \ge \delta_2 = \frac{( 1 - f(p))\delta ^{2 - q}}{2S}.$$

		It remains to consider the case $\Vert (\bar{u}_k, \bar {v}_k)\Vert \leq 1$. Under these conditions, by using Proposition \ref{tilde Longe de zero}, we observe that  
		$$\Lambda_n(\bar{u}_k, \bar{v}_k) \ge \frac{( 1 - f(p))\Vert (\bar{u}_k, \bar{v}_k)\Vert^2}{2S} \geq \delta_3 = \frac{( 1 - f(p))\delta^2}{2S}.$$
		As a consequence, we infer that
		$\Lambda_n(\bar{u}_k, \bar{v}_k) \ge \min\{\delta_2, \delta_3\} > 0$ for each $ k \in \mathbb{N}$.
		In particular, there exists $C_\delta > 0$ such that $$\lambda^*= \inf\limits_{(z, w)\in \mathcal{A}}\Lambda_n(z, w) = \displaystyle \lim_{k \to +\infty} \Lambda_n(\bar{u}_k, \bar{v}_k) \ge C_\delta > 0.$$ This ends the proof.
	\end{proof}

	\begin{prop}\label{esquema2}
		Suppose $(P)$, ($V_0$) and ($V_1$). Then there exists $C_\delta > 0$ such that  $\lambda_*\ge \bar{C}_\delta > 0$. 
	\end{prop}
	\begin{proof}
		The proof follows the same lines discussed in the proof of Proposition \ref{esquema}.	Firstly, by applying the estimates (\ref{2ky}), (\ref{Lambda_e de te}) together with (\ref{Q_e}) we obtain that  
		$$\Lambda_e(\bar{u}_k, \bar{v}_k) \ge C \min \left \{\Vert (\bar{u}_k, \bar{v}_k)\Vert^{2 - p},\Vert (\bar{u}_k, \bar{v}_k)\Vert^{2 - q}\right \}.$$
		Recall also that $\Lambda_e(\bar{u}_k, \bar{v}_k) \ge C_{\delta}> 0$ holds for some $C_{\delta} > 0$ and for any $k \in \mathbb{N}$. Hence we deduce that
		$$\lambda_*= \inf_{(z, w)\in \mathcal{A}}\Lambda_e(z, w) = \lim_{k \rightarrow +\infty} \Lambda_e(\bar{u}_k, \bar{v}_k) \ge \bar{C}_{\bar\delta} > 0.$$
		This completes the proof.
	\end{proof}
	
	\begin{prop}
		Suppose $(P)$, ($V_0$) and ($V_1$). Let $(\bar u, \bar v) \in \mathcal{A}$ be any minimizer for the functional $\Lambda_n$. Then $(\bar u, \bar v)$ is a critical point for the functional $\Lambda_n$. Furthermore, the function  satisfies $(z_1, z_2) = (t_n(\bar{u}, \bar{v})(\bar{u}, \bar{v}))$ weakly the following nonlocal elliptic problem:
		\begin{equation}\left\{\begin{array}{lll}\label{problema2} 
				2(-\Delta)^su + 2V_1(x)u =   \lambda^*p|u|^{p-2}u + \theta \alpha |u|^{\alpha - 2}u|v|^{\beta} \ \ em \ \ \mathbb{R^N}, \\
				2(-\Delta)^sv + 2V_2(x)v =   \lambda^*q|v|^{q-2}v + \theta \beta |u|^{\alpha}|v|^{\beta - 2}v \ \ em \ \ \mathbb{R^N},\\
				(u, v) \in X.
			\end{array}\right.
		\end{equation}
	\end{prop}
	\begin{proof}
		Since $\Lambda_n$ is attained by the function $(\bar{u}, \bar{v})\in \mathcal{A}$ we mention that 
		$$\lambda^* = \inf_{(u,v)\in \mathcal{A}} \Lambda_n(u,v) = \Lambda_n(\bar{u}, \bar{v}) = R_n(t_n(\bar{u}, \bar{v})(\bar{u}, \bar{v})).$$
		Recall also that $t_n(\bar{u}, \bar{v})$ is the maximum point for the function $t \mapsto Q_n(t) := R_n(tu, tv)$. In particular, we obtain that 
		\begin{equation}\label{equiv de deriv Qn' e Rn'} 
			0 = Q_n'(t_n(\bar{u}, \bar{v})) = (R_n)'(t_n(\bar{u},\bar{v})(\bar{u}, \bar{v}))(\bar{u}, \bar{v}).
		\end{equation}
		On the other hand, by using the fact that $(\bar{u}, \bar{v})$ is a critical point of $\Lambda_n$, we infer that
		\begin{eqnarray*} \label{der seg prob} 
			0&=&  \left (\Lambda_n \right )'(\bar{u},\bar{v})(w_1, w_2) = \left (R_n(t_n(\bar{u}, \bar{v})(\bar{u}, \bar{v})) \right )'(w_1, w_2)   \\
			&=&\left (R_n \right )'(t_n(\bar{u}, \bar{v})(\bar{u}, \bar{v})) \left \{\left [t_n'(\bar{u}, \bar{v})(w_1, w_2)\right ](\bar{u}, \bar{v}) +t_n(\bar{u}, \bar{v})(w_1, w_2) \right \} \\ 
			&= & t_n'(\bar{u}, \bar{v})(w_1, w_2)\left (R_n \right )'(t_n(\bar{u}, \bar{v})(\bar{u}, \bar{v}))(\bar{u}, \bar{v}) + t_n(\bar{u}, \bar{v})\left (R_n \right )'(t_n(\bar{u}, \bar{v})(\bar{u}, \bar{v}))(w_1, w_2).
		\end{eqnarray*}
		It follows from \eqref{equiv de deriv Qn' e Rn'} that $\left (R_n \right )'(t_n(\bar{u}, \bar{v})(\bar{u}, \bar{v}))(w_1, w_2) = 0$ holds for each $(w_1, w_2) \in \mathcal{A}$. Define the auxiliary function $(z_1, z_2) := t_n(\bar{u}, \bar{v})(\bar{u}, \bar{v}) \in \mathcal{A}$. Notice also that $$\lambda^*=R_n(z_1, z_2) = \frac{\displaystyle \Vert(z_1, z_2) \Vert^2 - \theta\int _\mathbb{R^N} \vert z_1 \vert^{\alpha}\vert z_2\vert^{\beta}dx}{\Vert z_1 \Vert^{p}_{p} + \Vert z_2 \Vert ^{q}_{q}}.$$
		Therefore, we obtain that
		\begin{eqnarray*} 
			0 & = & 2<(z_1, z_2), (w_1, w_2)> - \theta \displaystyle \alpha \int _\mathbb{R^N} \vert z_1 \vert^{\alpha - 2}z_1 w_1\vert z_2\vert^{\beta}dx - \theta\beta \int _\mathbb{R^N} \vert z_1 \vert^{\alpha}\vert z_2\vert^{\beta - 2}z_2 w_2dx \\
			&-&  \displaystyle\lambda^*\left ( p \int _\mathbb{R^N} \vert z_1 \vert^{p - 2}z_1 w_1 +  q \int _\mathbb{R^N} \vert z_2 \vert^{q - 2}z_2 w_2 dx \right ).
		\end{eqnarray*}
		holds for any $(w_1, w_2) \in X$. The last assertion says that $(z_1, z_2) \in X$ is a weak solution for the Problem (\ref{problema2}). This finished the proof.
	\end{proof}

	\begin{rmk}\label{lasqueira}
		Assume that $0 < \lambda < \lambda^*$. Then $\Lambda_n(u, v) > \lambda$ holds for each $(u, v)\in \mathcal{A}$. In particular, the equation $Q_n(t) = \lambda$ has exactly two distinct critical points for each $\lambda \in (0, \lambda^*)$.
	\end{rmk}
	Now, by using Remark \ref{lasqueira}, we obtain the following result: 
	
	\begin{prop}\label{tn-,tn+}
		Suppose (P), $(V_0)$ and $(V_1)$. Then for each $\lambda \in (0, \lambda^*)$ and $(u, v) \in \mathcal{A}$ the fibering map $ t \mapsto \gamma_{_\lambda}(t) = E_{_\lambda}(tu,tv)$ has exactly two distinct critical points $0 < t_n^+(u, v) < t_n(u, v) < t_n^-(u, v)$. Furthermore, we consider the following statements:
		\begin{itemize}
			\item [i)] The number $t_n^+(u, v)$ is a local minimum point for the fibering map $\gamma_{_\lambda}$ which satisfies $t_n^+(u, v)(u, v) \in \mathcal{N}^+_{_\lambda}$. Furthermore, the number $t_n^-(u, v)$ is a local maximum for the fibering map $\gamma_{_\lambda}$ which verifies $t_n^-(u, v)(u, v) \in \mathcal{N}^-_{\lambda}$.
			
			\item [ii)] The functions $(u, v) \mapsto t_n^+(u, v)$ and $(u, v) \mapsto t_n^-(u, v)$ are in $C^1(\mathcal{A}, \mathbb{R})$.
		\end{itemize}
	\end{prop}
	\begin{proof}
		Let $0 < \lambda < \lambda^*$ and $(u, v) \in \mathcal{A}$ be fixed. It is easy to see that $$\lambda < \lambda^* = \displaystyle\inf_{(z,w)\in \mathcal {A}}\Lambda_n(z, w) < \Lambda_n(u, v) = Q_n(t_n(u, v)) = R_n(t_n(u, v)(u, v)).$$ Recall also that $Q_n(t_n(u, v))= \max\limits_{t>0} Q_n(t)$. Therefore, $Q_n(t) = R_n(tu, tv) = \lambda$ admits exactly two roots which we denote by $t_n^+(u, v)$ and $t_n^-(u, v)$. It is not hard to see that $0 < t_n^+(u, v) < t_n(u, v) < t_n^-(u, v)$. Furthermore, we have that $t_n^+(u, v)$ and $t_n^-(u, v)$ are critical points for the fibering map $\gamma_{_\lambda}(t) = E_{_\lambda}(tu, tv)$, see Remark \ref{obs rel Rn eE'}. Notice also that $R_n(tu, tv)= \lambda$ if and only if $\gamma_{_\lambda}'(t) = E'_{_\lambda}(tu, tv)(tu, tv) = 0$. It is important to emphasize that $Q_n'(t_n^+(u, v)) > 0$ and $Q_n'(t_n^-(u, v)) < 0$. In view of Proposition \ref{relação entre R_n' e E''} we conclude that $E''_{_\lambda}(t_n^+(u, v)(u, v))(t_n^+(u, v)(u, v))^2 > 0$ and $E''_{_\lambda}(t_n^-(u, v)(u, v))(t_n^+(u, v)(u, v))^2 < 0$. Hence, we obtain that $t_n^+(u, v)(u, v) \in \mathcal{N}^+_\lambda$ and $t_n^-(u, v)(u, v) \in \mathcal{N}^-_\lambda$. This finishes the proof of item $(i)$.
		
		Now we shall prove the item $(ii)$. Firstly, we observe that $\lambda \in (0, \lambda^*)$. Hence, the equation $Q_n(t_n(u, v)) = R_n(t_n(u, v)(u, v))$ admits exactly two roots $(u, v) \in \mathcal{A}$ for each $(u, v) \in \mathcal{A}$. Namely, we obtain $t_n^+(u, v)$ and $t_n^-(u, v)$ such that $0 < t_n^+(u, v) < t_n(u, v) < t_n^-(u, v)$ where $t_n^+(u, v)(u, v) \in \mathcal{N}^+_\lambda$ and $t_n^-(u, v)(u, v) \in \mathcal{N}^-_\lambda$. Furthermore, for $\lambda \in (0, \lambda^*)$, we infer that $\mathcal{N}^0_{\lambda} = \emptyset$ and  $\mathcal{N}_{\lambda} = \mathcal{N}^+_\lambda \cup \mathcal{N}^-_{\lambda}$. 
		On the other hand, using the fact that $Q_n \in C^1(\mathbb{R}^+, \mathbb{R})$, we deduce that $Q_n'(t_n^+(u, v)) > 0$ and $Q_n'(t_n^-(u, v)) < 0$. Therefore, the desired result follows by using the Implicit Function Theorem, see \cite{DRABEK}. Under these conditions, we infer that the functional $(u, v) \mapsto t_n^+(u, v)$ belongs to $C^1(\mathcal{A}, \mathbb{R})$ for any $\lambda \in (0, \lambda^*)$. Similarly, the functional $(u, v) \mapsto t_n^-(u, v)$ belongs to $C^1(\mathcal{A}, \mathbb{R})$ for any $\lambda \in (0, \lambda^*)$. This finishes the proof. 
	\end{proof}
	\begin{prop}\label{tR_n'G' = E''} 
		Suppose $(P)$, ($V_0$) and ($V_1$). Assume also that $(u, v)\in \mathcal{A}$ satisfies $\lambda = R_n(tu,tv)$ for some $t > 0$. Consider the auxiliary functional $G: \mathcal{A}\rightarrow \mathbb{R}$ given by $G(u,v) = \Vert u \Vert ^{p}_{p} +\Vert v \Vert ^{q}_{q}$. Then, we obtain that $$\frac{d}{dt}R_n(tu, tv) = \frac{1}{t}\frac{E''_{_\lambda}(tu, tv)(tu, tv)^2}{G(tu, tv)}.$$ 
	\end{prop}
	\begin{proof}
		Firstly, we mention that
		\begin{equation}\label{der função auxiliar} 
			t\frac{\partial}{\partial t}G(tu, tv) = p\Vert tu \Vert^{p}_{p} + q\Vert tv \Vert^{q}_{q}. 
		\end{equation}
		It is not hard to see that 
		$$G(tu, tv)R_n(tu, tv) = t^2\Vert (u, v) \Vert^2 - \theta t^{\alpha + \beta}\int\limits_{\mathbb{R}^N} \vert u\vert^\alpha\vert v\vert^\beta dx.$$
		Therefore, 
		\begin{equation}\label{derivaçao de eq} 
			\frac{d}{dt}G(tu, tv)R_n(tu, tv) + G(tu, tv) \frac{d}{dt}R_n(tu, tv)= 2t\Vert (u, v) \Vert^2 -\theta(\alpha + \beta)t^{\alpha +\beta -1}\int\limits_{\mathbb{R}^N} \vert u\vert^\alpha\vert v\vert^\beta dx.
		\end{equation}
		In view of hypothesis $R_n(tu, tv) = \lambda$ and (\ref{derivaçao de eq}) we infer that 
		$$t\frac{d}{dt}G(tu, tv)\lambda + tG(tu, tv) \frac{d}{dt}R_n(tu, tv)= 2\Vert (tu, tv) \Vert^2 -\theta(\alpha + \beta)\int\limits_{\mathbb{R}^N} \vert tu\vert^\alpha\vert tv\vert^\beta dx.$$
		Now, taking into account (\ref{der função auxiliar}), we deduce that
		\begin{eqnarray*} \label{der seg prob} 
			tG(tu, tv)\frac{d}{dt} R_n(tu, tv) & = & 2 \Vert(tu, tv) \Vert^2 - \theta (\alpha + \beta)\int\limits_{\mathbb{R}^N} \vert tu\vert^\alpha\vert tv\vert^\beta dx - \lambda (p \Vert tu\Vert_p^p + q\Vert tv \Vert_q^q)  =  E''_{_\lambda}(tu, tv)(tu, tv)^2.
		\end{eqnarray*}
		Under these conditions, we were able to show that
		\begin{equation}\label{sinal Rn' e E''} 
			\frac{d}{dt} R_n(tu, tv) = \frac{E''_{_\lambda}(tu, tv)(tu, tv)^2 }{tG(tu, tv)}. 
		\end{equation}
		This ends the proof.
	\end{proof}
	Now, by using Proposition \eqref{tR_n'G' = E''},  we obtain the following result:
	\begin{rmk} \label{relação entre R_n' e E''}
		Suppose $(P)$, ($V_0$) and ($V_1$). Assume also that $(u, v)\in \mathcal{A}$ satisfies $\lambda = R_n(tu,tv)$ for some $t > 0$. Then we obtain that 
		$\frac{d}{dt}R_n(tu, tv)>0$ if, and only if, $E''_{_\lambda}(tu, tv)(tu, tv)^2>0$. Furthermore, $\frac{d}{dt}R_n(tu, tv)<0$ if, and only, if $E''_{_\lambda}(tu, tv)(tu, tv)^2<0$. In the same way, $\frac{d}{dt}R_n(tu, tv) = 0$ if, and only if, $E''_{_\lambda}(tu, tv)(tu, tv)^2 = 0$.
	\end{rmk}
	
	\begin{prop}\label{tR_e'G = E'} 
		Suppose $(P)$, ($V_0$) and ($V_1$). Assume also that $(u, v)\in \mathcal{A}$ satisfies $\lambda = R_e(tu,tv)$ for some $t > 0$. Define the auxiliary functional $G: \mathcal{A}\rightarrow \mathbb{R}$ given by $G(u,v) = \Vert u \Vert ^{p}_{p} +\Vert v \Vert ^{q}_{q}$. Then $$\frac{d}{dt}R_e(tu, tv) = \frac{1}{t}\frac{E'_{_\lambda}(tu, tv)(tu, tv)}{G(tu, tv)}.$$ 
	\end{prop}
	\begin{proof}
		Using (\ref{der ener}) together with the identity $R_e(tu, tv) = \lambda$ we obtain the following identities
		\begin{eqnarray*}
			\frac{E'_{_{\lambda}}(tu, tv)(tu, tv)}{G(tu, tv)} &=&\frac{\left ( t^2\Vert (u, v)\Vert^2 - \theta t^{\alpha + \beta} \int\limits_{\mathbb{R}^N} \vert u\vert^\alpha\vert v\vert^\beta dx  \right  )\left ( \frac{t^{p}}{p}\Vert u \Vert^{p}_{p} + \frac{t^{q}}{q}\Vert v \Vert^{q}_{q}  \right  ) }{\left ( \frac{t^{p}}{p}\Vert u \Vert _{p}^{p} + \frac{t^{q}}{q}\Vert v \Vert_{q}^{q} \right )^2}\\ 
			&-&\frac{\left ( \frac{1}{2}t^2\Vert(u, v) \Vert^2 - \frac{\theta}{\alpha + \beta}t^{\alpha + \beta} \int\limits_{\mathbb{R}^N} \vert u\vert^\alpha\vert v\vert^\beta dx  \right )\left ( t^{p}\Vert u \Vert ^{p}_{p} + t^{q}\Vert v \Vert ^{q}_{q}  \right )}{\left ( \frac{t^{p}}{p}\Vert u \Vert_{p}^{p} + \frac{t^{q}}{q}\Vert v \Vert_{q}^{q} \right )^2}\\ 
			& = & t\frac{d}{dt}R_e(tu, tv).
		\end{eqnarray*}
		This ends the proof.
	\end{proof}
	In particular, using Proposition \ref{tR_e'G = E'}, we obtain the following result:
	\begin{rmk}\label{signal R_e' = signal E'} 
		Suppose $(P)$, ($V_0$) and ($V_1$). Assume also that $(u, v)\in \mathcal{A}$ satisfies $\lambda = R_e(tu,tv)$ for some $t > 0$. Then we obtain that $\frac{d}{dt}R_e(tu, tv)>0$ if and only if $E'_{_\lambda}(tu, tv)(tu, tv)>0$. Furthermore, $\frac{d}{dt}R_e(tu, tv)<0$ if and only if $E'_{_\lambda}(tu, tv)(tu, tv)<0$. In the same manner, $\frac{d}{dt}R_e(tu, tv) = 0$ if and only if $E'_{_\lambda}(tu,  tv)(tu, tv) = 0$.
	\end{rmk}
	\begin{prop}
		Suppose $(P)$, ($V_0$) and ($V_1$). Then $\Lambda_n(u, v) > \Lambda_e(u, v)$ holds for each $(u, v) \in \mathcal{A}$. As a consequence, we obtain that $0 < \lambda_* <\lambda^* < + \infty$.
	\end{prop}
	\begin{proof}
		Firstly, we observe that  $Q_n(t_n(u, v)) > Q_n(t)$ holds for each $t \neq t_n(u,v)$. In light of Proposition \ref{lema_naruprop} we obtain that
		$$\Lambda_n(u, v) - \Lambda_e(u, v) = Q_n(t_n(u, v)) - Q_e(t_e(u, v)) > Q_n(t_e(u, v)) - Q_e(t_e(u, v)) = 0$$
		Recall also that $\frac{dQ_e(t)}{dt} = 0$ only for  $t = t_e(u, v)$, see Remark \ref{signal R_e' = signal E'}. In particular, we obtain that
		$\Lambda_e(u, v) < \Lambda_n(u, v)$ holds for each $(u,v) \in \mathcal{A}$. Using the last assertion we infer that
		$$\lambda^* = \inf_{(u, v)\in \mathcal{A}} \Lambda_n(u,v) = \Lambda_n(u^*, v^*)> \Lambda_e(u^*, v^*) \ge \inf_{(z, w)\in \mathcal{A}}\Lambda_e(z, w)  = \lambda_*.$$
		Here was used the fact that $\lambda^*$ is attained by the function $(u^*, v^*) \in \mathcal{A}$. Hence, $0 < \lambda_* < \lambda^* < \infty$. The finishes the proof.
	\end{proof}
	\begin{prop}
		Suppose $(P)$, ($V_0$), ($V_1$) and $\lambda \in ( 0, \lambda^*)$. Then $\mathcal{N}^0_{\lambda} = \emptyset$ holds.
	\end{prop}
	\begin{proof}
		The proof follows arguing by contradiction assuming that $\mathcal{N}^0_\lambda \ne \emptyset$. Let $(u, v) \in \mathcal{N}^0_\lambda$ be a fixed function. Therefore, $t_n(u, v) = 1$.  Notice also that $\lambda < \lambda^*$ and $$\lambda < \lambda^* = \inf_{(z, w)\in \mathcal{A}}\Lambda_n(z, w) \le \Lambda_n(u, v) =R_n(t_n(u, v)(u, v))  = R_n(u, v) = \lambda.$$
		This is a contradiction proving that $\mathcal{N}^0_{\lambda} = \emptyset$ for any $\lambda \in (0, \lambda^*)$. This ends the proof. 
	\end{proof}
	\begin{prop}\label{N0dif vazio}
		Suppose $(P)$, ($V_0$), ($V_1$) and $\lambda = \lambda^*$. Then $\mathcal{N}^0_{\lambda}\ne \emptyset$ holds.
	\end{prop}
	\begin{proof}
		Let $(\tilde{u}, \tilde{v}) \in \mathcal{A}$ be fixed such that $\lambda^* = \Lambda_n(\tilde{u}, \tilde{v})$. Since $\lambda = \lambda^*$ it follows that $\lambda = \Lambda_n(\tilde{u}, \tilde{v}) = R_n(t_n(\tilde{u}, \tilde{v})(\tilde{u}, \tilde{v})) = max_{t>0} R_n(t (\tilde{u}, \tilde{v}))$. Hence, $t_n(\tilde{u}, \tilde{v})(\tilde{u}, \tilde{v}) \in \mathcal{N}_\lambda$. Furthermore, $\frac{d}{dt}R_n(t\tilde{u}, t\tilde{v}) = 0$ for $t = t_n(\tilde{u}, \tilde{v})$. Now, by using Proposition \ref{relação entre R_n' e E''},  we deduce that $E''(t_n(\tilde{u}, \tilde{v})(\tilde{u}, \tilde{v})) = 0$. Therefore $t_n(\tilde{u}, \tilde{v})(\tilde{u}, \tilde{v}) \in \mathcal{N}^0_\lambda$. The last assertion ensures that  $\mathcal{N}^0_\lambda \ne \emptyset$ holds true for $\lambda = \lambda^*$. This finishes the proof.
	\end{proof}
	\begin{rmk}
		It is important to emphasize that Proposition \ref{N0dif vazio} shows that $\mathcal{N}^0_{\lambda} \ne \emptyset$ for $\lambda = \lambda^*$. More generality, we can show that $t_n(u, v)(u, v) \in \mathcal{N}^0_{\lambda}$ whenever $\lambda = \Lambda_n(u, v)$. In other words, $\lambda = \lambda^*$ is the first positive value such that $\mathcal{N}^0_{\lambda} \ne \emptyset$.
	\end{rmk}
	\begin{prop}\label{Nehari - longe de zero}
		Suppose $(P)$, ($V_0$) and ($V_1$). Assume also that $\lambda \in ( 0, \lambda^*)$. Then there exists $C > 0$ such that $\Vert (u, v) \Vert \ge C > 0$ for each $(u, v) \in \mathcal{N}^-_{\lambda}$.
	\end{prop}
	\begin{proof}
		Let $(u, v) \in \mathcal{N}^-_\lambda$ be a fixed function. Hence, we deduce that
		\begin{equation}\label{equação de Nehari} 
			\lambda \Vert u \Vert ^{p}_{p} + \lambda \Vert v \Vert ^{q}_{q}= \Vert (u, v) \Vert ^2  - \theta\int\limits_\mathbb{R^N} |u|^\alpha |v|^\beta dx
		\end{equation}
		and
		\begin{eqnarray*}\label{equação de Nehari -}
			2\Vert (u, v) \Vert ^2 - \theta ( \alpha +\beta) \int\limits_\mathbb{R^N} |u|^\alpha |v|^\beta dx & < & \lambda p \Vert u \Vert ^{p}_{p} + \lambda q \Vert v \Vert ^{q}_{q}.\\
		\end{eqnarray*}
		Without loss of generality we assume that $ q \leq p$. Under these conditions,  by using \eqref{equação de Nehari}, we obtain that 
		\begin{equation}\label{ineq norma e integ} 
			(2 - q)\Vert (u, v) \Vert ^2 < \theta (\alpha + \beta - q)\int\limits_\mathbb{R^N} |u|^\alpha |v|^\beta dx.
		\end{equation}
		According to Lemma \ref{chave} we obtain also that $$\int\limits_\mathbb{R^N} |u|^\alpha |v|^\beta dx \le S^{\alpha + \beta}_{\alpha + \beta}\Vert(u, v)\Vert^{\alpha + \beta}.$$
		In particular, by using \eqref{equação de Nehari} and the last estimate, we see that
		\begin{eqnarray*}\label{2ª equação de Nehari -}
			\Vert(u, v)\Vert & \ge &  \left (\frac{2 - \theta  q}{\theta [\alpha + \beta - q]S^{\alpha + \beta}_{\alpha + \beta}}\right )^{\frac{1}{\alpha + \beta - 2}}.
		\end{eqnarray*}
		Hence, the desired result follows for $C = C(p, q, \theta, \alpha , \beta) = \left (\frac{2 - \theta q}{\theta [\alpha + \beta - q]S^{\alpha + \beta}_{\alpha + \beta}  }\right )^{\frac{1}{\alpha + \beta - 2}}$.
	\end{proof}
	\begin{rmk}\label{limdifsemtriv}
		It is worthwhile to mention that \eqref{ineq norma e integ} implies that for each sequence $(u_k, v_k) \in \mathcal{N}^-_{\lambda}$ such that $(u_k, v_k) \rightharpoonup (u, v)$ we obtain that $u \ne 0$ and $v \ne 0$. Here was used the Lemma \ref{chave} which provide us the compact embedding.
	\end{rmk} 
	\begin{prop}\label{chave2}
		Suppose $(P)$, ($V_0$) and ($V_1$). Let $(u_k, v_k)$ a sequence in $\mathcal{N}^-_{\lambda}$ such that $(u_k, v_k) \rightharpoonup (u, v)$ in $X$. Then there exists $\delta_C > 0$ such that 
		$$\int_{\mathbb{R}^N} |u|^\alpha|v|^\beta dx \ge \delta_C > 0.$$
	\end{prop}
	\begin{proof}
		Initially, we observe that $(u_k, v_k) \in \mathcal{N}^-_{\lambda}$. Now, using (\ref{ineq norma e integ} ) together wit Proposition \ref{Nehari - longe de zero}, we infer that
		$$\int\limits_{\mathbb{R}^N} \vert{u}_k\vert^\alpha\vert{v}_k\vert^\beta dx > (2 - q) C > 0.$$
		Hence, by using the Dominated Convergence Theorem, we also obtain that $$\int\limits_{\mathbb{R}^N} \vert{u}\vert^\alpha\vert{v}\vert^\beta dx \ge \delta_C = (2 - q) C > 0.$$
		This completes the proof.
	\end{proof}
	\begin{prop}\label{soluçao N-}
		Suppose $(P)$, ($V_0$) and ($V_1$). Assume also that $\lambda \in ( 0, \lambda^*)$. Then $\mathcal{N}^-_\lambda$ is a closed which is away from zero. Furthermore, there exists $(u, v) \in \mathcal{N}^-_\lambda$ such that $$C_{\mathcal{N}^-_\lambda} = \inf\limits_{(z, w) \in \mathcal{N}^-_\lambda}E_{_\lambda}(z, w) = E_{_\lambda}(u, v)$$ where $(u, v)$ is a weak solution to the System (\ref{sistema Principal}).
	\end{prop}
	\begin{proof}
		Let $(u_k, v_k)$ be a minimizer sequence for $E_{_\lambda}$ in $\mathcal{N}^-_\lambda$. In particular, we know that $E_{_\lambda}(u_k, v_k) \to C_{\mathcal{N}^-_\lambda}$ as $k \to \infty$. Recall also that $$\lambda < \lambda^* = \inf\limits_{(z, w) \in \mathcal{A}}\Lambda_n(z, w) < \Lambda_n(z, w)$$ holds true for each $(z, w) \in \mathcal{A}$. Furthermore, we observe that $(u_k, v_k)$ is a bounded sequence, see Proposition \ref{E coersiva}.
		Hence, there exists $(u,v) \in X$ such that $(u_k, v_k) \rightharpoonup (u, v)$ in $X$. Now, we claim that $(u, v) \ne (0, 0)$. The proof for this claim follows arguing by contradiction. Assume that $(u, v) = (0, 0)$. Hence, by using Lemma \ref{chave} and Proposition \ref{chave2}, we obtain $$0 < \delta_C \leq \int\limits_{\mathbb{R}^N} \vert{u}_k\vert^\alpha\vert{v}_k\vert^\beta dx \to 0.$$
		This is a contradiction proving the desired claim. It remains to prove that $(u_k, v_k) \to (u, v)$ in $X$. Once again the proof follows arguing by contradiction. Assume that $(u_k, v_k)$
		does not strong converge to $(u, v)$ in $X$. Now, by using Proposition \ref{tn-,tn+}, there exists $t_n^-(u, v) >0$ such that $t_n^-(u, v)(u, v) \in \mathcal{N}^-_\lambda$. The last assertion implies that $C_{\mathcal{N}^-_\lambda}\le E_{_\lambda}(t_n^-(u, v) (u, v))$. Notice also that $E'_{_\lambda}(u_k, v_k)(u_k, v_k) = 0$ and $E''_{_\lambda}(u_k, v_k)(u_k, v_k)^2 < 0$ hold for each $k \in \mathbb{N}$. Furthermore, by using Proposition \ref{Nehari - longe de zero} and Proposition \ref{chave2}, we observe that $$\Vert (u_k, v_k) \Vert \ge C > 0 \,\, \mbox{and} \,\, \int\limits_{\mathbb{R^N}}\vert u_k\vert^{\alpha}\vert v_k \vert ^{\beta} \ge \delta_c > 0$$ hold for each $k \in \mathbb{N}$. Here we emphasize that $t_n^-(u_k, v_k) = 1$ holds for each $t > 0$. Recall also that $(u,v) \mapsto E_\lambda(u,v)$ is weakly lower semicontinuous. Hence, we obtain that 
		\begin{equation}\label{cn-} 
			C_{\mathcal{N}^-_\lambda} < \liminf\limits_{k \to \infty}E_{_\lambda}(t_n^-(u, v)(u_k, v_k)).
		\end{equation}
		It is not hard to see that $$E'_{_\lambda}(u, v)(u, v) < \liminf\limits_{k \to \infty }E'_{_\lambda}(u_k, v_k)(u_k, v_k) = 0, E''_{_\lambda}(u, v)(u, v)^2 < \displaystyle\liminf_{k \to \infty} E''_{_\lambda}(u_k, v_k)(u_k, v_k)^2 \le 0$$ and $$\gamma_{_\lambda}'(t) =E'_{_\lambda}(tu, tv)(u, v), E'_{_\lambda}(u, v)(u, v) =  \left. \frac{\partial E_{_\lambda}}{\partial t} (tu, tv) \right|_{t=1} = \left. \frac{\partial \gamma_{_\lambda} }{\partial t}(t) \right|_{t=1}.$$  
		Under these conditions, we obtain that $t_n^-(u, v) < 1$. Furthermore, by using the fact that $t_n^-(u_k, v_k) = 1$, we infer that
		$$E_{_\lambda}(tu_k, tv_k) \le \max_{t\in [ t_n^+(u_k, v_k), t_n^-(u_k, v_k)]} E_{_\lambda}(tu_k, tv_k) = E_{_\lambda}(t_n^-(u_k, v_k)(u_k, v_k)) = E_{_\lambda}(u_k, v_k).$$
		Moreover, we observe that
		\begin{equation}\label{E(tuk,tvk)} 
			E_{_\lambda}(tu_k, tv_k) \le E_{_\lambda}(u_k, v_k), \,\, t \in [t_n^+(u_k, v_k), 1].
		\end{equation}
		Now, by using (\ref{cn-}) and (\ref{E(tuk,tvk)}), we also obtain that
		$$C_{\mathcal{N}^-_\lambda} < \liminf\limits_{k \to \infty}E_{_\lambda}(t_n^-(u, v)(u_k, v_k)) \le \liminf\limits_{k \to \infty}E_{_\lambda}(u_k, v_k) = C_{\mathcal{N}^-_\lambda}.$$
		This is a contradiction proving that $$\Vert (u, v) \Vert^2 = \liminf\limits_{k \to \infty}\Vert (u_k, v_k) \Vert^2.$$ Now, by using the fact that $X$ is Hilbert space together with the fact that $(u_k, v_k)\rightharpoonup (u, v)$ and $\Vert(u_k, v_k) \Vert \to \Vert (u,v)\Vert$, we deduce $(u_k, v_k) \to (u, v)$ in $X$.
		
		At this stage, we observe that the energy functional $E_{_\lambda}$ belongs to $C^1(X, \mathbb{R})$. It is important to stress that the functional $(u,v) \mapsto E''_{_\lambda}(u, v)(u, v)^2$ is well defined and continuous.  As a consequence, we obtain that $$C_{\mathcal{N}^-_\lambda} = \lim\limits_{k \to \infty} E_{_\lambda}(u_k, v_k) = E_{_\lambda}(u, v) \,\, \mbox{and} \,\, E'_{_\lambda}(u, v)(u, v) = 0.$$ 
		Hence, $(u, v) \in \mathcal{N}_\lambda$.  Similarly, we obtain that $E''_{_\lambda}(u, v)(u, v)^2 \le 0$ holds for each $\lambda \in (0,  \lambda^*)$.  Recall also that $\mathcal{N}^0_\lambda = \emptyset$ for each $\lambda \in (0,  \lambda^*)$. Under these conditions, we obtain that $E_{_\lambda}''(u, v)(u, v)^2 < 0$. The last assertion implies that $(u, v) \in \mathcal{N}^-_\lambda$. Hence, $C_{\mathcal{N}^-_\lambda}$ is attained in the set $\mathcal{N}^-_\lambda$. It remains to ensure that $(u, v)$ is a critical point of $E_{_\lambda}$ in $X$, that is, $E'_{_\lambda}(u, v)(\phi, \psi) = 0$ holds for each $(\phi, \psi) \in X$. The last assertion implies that $(u, v)$ is a weak nontrivial for the System (\ref{sistema Principal}). The main ingredient here is to apply the Lagrange Multiplier Theorem, see \cite[Theorem 7.8.2]{DRABEK}. Notice that minimum for the energy functional $E$ restricted to $\mathcal{N}_\lambda^-$ is attained by a function  $(u, v) \in \mathcal{N}^-_\lambda$. Define $R: \mathcal{A} \to \mathbb{R}$ which is given by $R(u, v) = E'_{_\lambda}(u,v)(u, v)$. It is not hard to see that $R$ belongs to $C^1(\mathcal{A}, \mathbb{R})$. Furthermore, we observe that $R^{-1}(0) = \mathcal{N}^-_\lambda \cup \mathcal{N}^+_\lambda $. Here we also mention that $\mathcal{A}$ is an open cone set in $X$.
		Recall also that $R'(u, v) (\phi, \psi) = E''_{_\lambda}(u, v)((\phi, \psi)(u, v)) + E'_{_\lambda}(u, v)(\phi, \psi)$ is verified for each $(\phi, \psi) \in X$. Now, using the Lagrange Multiplier Theorem, there exists $\theta \in \mathbb{R}$ such that
		\begin{equation}\label{Teo mult lag} 
			E'_{_\lambda}(u, v)(\phi, \psi) = \theta R'(u, v)(\phi, \psi), \,\, (\phi, \psi) \in X.
		\end{equation}
		In particular,  for $(\phi, \psi)=(u, v)$, we infer that
		\begin{equation}\label{R'(u,v)(u, v)} 
			R'(u, v)(u, v) = E''_{_\lambda}(u, v)(u, v)^2.
		\end{equation}
		Notice that $(u,v) \in \mathcal{N}_\lambda^-$ implies that
		$R'(u, v)(u, v) < 0$. Hence, by using (\ref{Teo mult lag}), we see that $0 = \theta R'(u, v)(u, v)$. Therefore, $\theta = 0$ which implies that 
		$$E'_{_\lambda}(u, v)(\phi, \psi) = 0, \,\, (\phi, \psi) \in X.$$
		Under these conditions, we observe that $C_{\mathcal{N}^-_\lambda} = E_{_\lambda}(u, v), \, \,(u, v) \in \mathcal{N}^-_\lambda$ and $E'_{_\lambda}(u, v)(\phi, \psi) = 0, (\phi, \psi) \in X.$ This ends the proof.
	\end{proof}
	\begin{prop}\label{CN+ neg}
		Suppose $(P)$, ($V_0$) and ($V_1$). Assume that $\lambda \in (0, \lambda^*)$. Then $C_{\mathcal{N}^+_\lambda\cap \mathcal{A}}= E_{_\lambda}(u, v) < 0$.
	\end{prop}
	\begin{proof}
		Let $\lambda \in (0, \lambda^*)$ be fixed. Now, by using Remark \ref{lema_naru}, we see that $R_e(tu, tv) < R_n(tu, tv)$ holds for each  $t \in (0, t_e(u, v))$ where $(u, v) \in \mathcal{A}$. Furthermore, we observe that $t_n^+ (u, v) < t_n (u, v) < t_e(u, v)$. In particular, $R_e(t_n^+(u, v)(u, v)) < R_n(t_n^+(u, v)(u, v)) = \lambda$. Now, by using Remark \ref{obs rel Re e E}, we obtain that $E_{_\lambda}(t_n^+(u, v)(u, v)) < 0$. Recall also that $t_n^+(u, v)(u, v) \in \mathcal{N}^+_\lambda\cap \mathcal{A}$. Under these conditions, we obtain that $$C_{\mathcal{N}^+_\lambda\cap \mathcal{A}} = \displaystyle \inf_{(z, w) \in \mathcal{N}^+_\lambda \cap \mathcal{A}} E_{_\lambda}(z, w) \le E_{_\lambda}(t_n^+(u, v) (u, v)) < 0.$$ This ends the proof.
	\end{proof}
	
	\begin{prop}\label{soluçao N^+}
		Suppose $(P)$, ($V_0$) and ($V_1$). Assume that $\lambda \in (0, \lambda^*)$. Let $(u_k, v_k)\in \mathcal{N}^+_\lambda \cap \mathcal{A}$ be a minimizer sequence for the energy functional $E_{_\lambda}$ restricted to $\mathcal{N}^+_\lambda\cap \mathcal{A}$. Then there exists $(u, v) \in \mathcal{A}$ such that $(u_k, v_k) \to (u, v)$ in $X$ where $(u, v) \in \mathcal{N}^+_\lambda\cap \mathcal{A}$. Furthermore, $C_{\mathcal{N}^+_\lambda \cap \mathcal{A}}= E_{_\lambda}(u, v)$ and $(u, v)$ is critical point for the functional $E_{_{\lambda}}$. 
	\end{prop}
	\begin{proof}
		Assume that $\lambda \in (0, \lambda^*)$. Let $(u_k, v_k)\in \mathcal{N}^+_\lambda \cap \mathcal{A}$ be a minimizer sequence for the energy functional $E_{_\lambda}$ restricted to $\mathcal{N}^+_\lambda\cap \mathcal{A}$. Hence, $(u_k,v_k)$ is a bounded sequence and there exists $(u, v) \in X$ such that $(u_k, v_k)\rightharpoonup (u, v)$ in $X$, see Proposition \ref{E coersiva}. Now, we observe that $(u,v) \neq (0,0)$. This can be done arguing by contradiction. Assume by contradiction that $(u_k, v_k)\rightharpoonup (0, 0)$ in $X$. Under these conditions, by using Lemma \ref{chave} and \eqref{Nehari}, we obtain that $\|(u_k, v_k)\| \to 0$ as $k \to \infty$. Therefore, by using Proposition \ref{CN+ neg}, we obtain that $0 > C_{\mathcal{A}\cap \mathcal{N}^+_\lambda} = \lim_{k \to \infty} E_\lambda(u_k,v_k) = 0$. This is a contradiction proving that $(u,v) \neq (0,0)$.
		
		At this stage, we shall assume that $(u, v) \in \mathcal{A}$. Firstly, we show that $(u_k, v_k) \to (u, v)$ in $X$. Notice that the functionals $(u,v) \mapsto E_\lambda(u,v)$ and $(u,v) \mapsto E_\lambda'(u,v) (u,v)$ are weakly lower semicontinuous. Hence,  we obtain that $$E'_{_{\lambda}}(u, v)(u, v) \le \lim \inf\limits_{k \to \infty} E'_{_{\lambda}}(u_k, v_k)(u_k, v_k) = 0, \,\, E_{_{\lambda}}(u, v) \le \lim \inf\limits_{k \to \infty} E_{_{\lambda}}(u_k, v_k) = C_{\mathcal{N}^+_{\lambda}}.$$
		In particular, we obtain that that $t_n^+(u, v) \ge 1$. Recall also the fact that fibering map 
		$\gamma_{_\lambda}(t) = E_\lambda(tu,tv)$ is decreasing for each $t \in (0, t_n^+(u, v))$. Notice also that $t_n^+(u, v)(u, v) \in \mathcal{N}^+_{\lambda}$. Using the last assertions we see that 
		$$C_{\mathcal{N}^+_{\lambda}} \le E_{_{\lambda}}(t_n^+(u, v)(u, v)) \le E_{_{\lambda}}(u, v) \le \liminf\limits_{k \to \infty} E_{_{\lambda}}(u_k, v_k) = C_{\mathcal{N}^+_{\lambda}}.$$
		As a consequence, we obtain that  $$t_n^+(u, v) = 1, \,\,E_{_{\lambda}}(u, v) = \lim_{k \to \infty} E_{\lambda}(u_k, v_k).$$
		Therefore, we deduce that $(u_k, v_k) \to (u, v)$ in $X$. Here we observe that 
		the functionals $(u,v) \mapsto E_{_{\lambda}}(u,v)$, $(u,v) \mapsto E'_{_{\lambda}}(u,v)(u,v)$ and $(u,v) \mapsto E''_{_{\lambda}}(u,v)(u,v)^2$ are continuous. Hence,
		$(u,v) \in \mathcal{N}^+_{\lambda} \cap \mathcal{A}$ and $E_{_\lambda}(u, v) = C_{\mathcal{N}^+_{\lambda}\cap \mathcal{A}}$.
		
		Now, using the same ideas discussed in the proof of Proposition \ref{soluçao N-}, we can apply the Lagrange Multiplier Theorem which implies that $(u, v)$ is the critical point for the energy functional $E_\lambda$.

		It remains to prove that  the infimum for the energy functional $E_\lambda$ restricted to  $\mathcal{N}^+_{_{\lambda}}$ is never attained by a semitrivial solution of the type $(u,0)$ or $(0,v)$. More generally, we shall prove that any minimizer $(u,v) \in X$ for the functional $E_\lambda$ restricted to the set $\mathcal{A} \cap \mathcal{N}^+_\lambda$ belongs to $\mathcal{A}$. Without loss of generality we assume that $u \geq 0$ and $v \geq 0$ in $\mathbb{R}^N$. Let us assume by contradiction that there exists some minimizer $(u,v) \in X \setminus \mathcal{A}$ for the functional $E_\lambda$ restricted to $\mathcal{A}\cap \mathcal{N}_\lambda^+$. Hence,  $$\int\limits_{\mathbb{R}^N} \vert u \vert^\alpha\vert v\vert^\beta dx = 0$$ which implies that $u v = 0$ almost every in $\mathbb{R}^N$. Recall also that $E'_{_{\lambda}}(u, v)(\phi, \psi) = 0$ holds true for each $(\phi, \psi) \in X$.  In particular, we obtain that 
		\begin{equation}
			E'_{_{\lambda}}(u, v)(\phi, 0) = 0, \,\,E'_{_{\lambda}}(u, v)(0, \psi) = 0 .
		\end{equation}
		Under these conditions,  by using the fact that $u \ge 0$ and $v \ge 0$ in $\mathbb{R}^N$, we obtain that $u$ and $v$ are solutions for the following two scalar problems:
		\begin{equation}
			(-\Delta)^su +V_1(x)u =   \lambda u^{p-1}, (-\Delta)^sv +V_2(x)v =   \lambda v^{q-1} \,\, \mbox{in} \,\, \mathbb{R}^N.
		\end{equation}
		It is important to stress that $u, v \in L^\infty(\mathbb{R}^N)$ is verified, see Proposition \ref{regular} ahead. In particular, we obtain that $u, v \in C^{0,\alpha}(\mathbb{R}^N)$ holds for some $\alpha \in (0,1)$. Now, 	applying the Strong Maximum Principle \cite[Theorem 1.2]{Maxfor}, we infer that 
		\begin{equation}\left\{\begin{array}{ll} 
				u \equiv 0  \ \hbox{or} \ u > 0 \;\; \mbox{in}\;\; \mathbb{R}^N,  \\
				v \equiv 0 \ \hbox{or} \ v > 0 \;\; \mbox{in}\;\; \mathbb{R}^N.
			\end{array}\right.
		\end{equation}
		Notice also that $(u, v) \neq (0, 0)$. Let us assume by contradiction that  $u > 0$ and $v \equiv 0$ in $\mathbb{R}^N$. Consider the function $t \mapsto E_\lambda(tz,tw)$ where $(z,w) \in X$. It is easy to verify that
		\begin{eqnarray*}
			\frac{d}{dt}E_{_{\lambda}}(tz, tw) & = & E'(tz, tw)(tz, tw)  =  t \Vert(z, w) \Vert^2 - \lambda t^{p - 1} \Vert z \Vert_p^p - \lambda t^{q - 1} \Vert w \Vert_q^q - \theta t^{\alpha + \beta - 1}\int\limits_{\mathbb{R}^N} \vert z \vert^\alpha\vert w\vert^\beta dx.
		\end{eqnarray*}
		Now, using the function $sw$ instead of $w$ where $s > 0$, we also obtain 
		\begin{eqnarray*}
			\frac{d}{dt}E_{_{\lambda}}(tz, tsw)  & = & E'(tz, tsw)(z, sw) =  t \Vert(z, sw) \Vert^2 - \lambda t^{p - 1} \Vert z \Vert_p^p - \lambda t^{q - 1} \vert s \vert^q \Vert w \Vert_q^q - \theta t^{\alpha + \beta - 1}\vert s \vert^{\beta}\int\limits_{\mathbb{R}^N} \vert z \vert^\alpha\vert w\vert^\beta dx.
		\end{eqnarray*}
		Recall that $(u, 0) \in \mathcal{N}^+_{_{\lambda}}$ and $C_{\mathcal{N}^+_{_{\lambda}}} = E_{_\lambda}(u, 0)$. As a consequence, $$\Vert u \Vert ^2 = \lambda \Vert u \Vert_p^p,
		E''_{_{\lambda}}(u, 0)(u, 0)^2 > 0$$
		where $E''_{_{\lambda}}(u, 0)(u, 0)^2 = 2\Vert u \Vert^2 - \lambda p \Vert u \Vert_p^p = \lambda( 2 - p) \Vert u \Vert _p^p$. Now, we define the function $F : (0, \infty) \times (-\epsilon, \epsilon) \to \mathbb{R}$ of $C^1$ class given by $F(t, s) = \frac{d}{dt}E_{_{\lambda}}(tz, tsw)$ for each $(z, w) \in X$ fixed. It is not hard to verify that 
		\begin{eqnarray}\label{equação de Nehari -}
			\frac{d}{dt}F(t, s) & = & \Vert (z, sw) \Vert ^2 - \lambda(p - 1) t^{p- 2}\Vert z \Vert_p^p - \lambda (q - 1) t^{q-2}\vert s \vert^q \Vert w\Vert _q^q \nonumber \\
			& -&  \theta(\alpha + \beta - 1) t^{\alpha + \beta - 2} \vert s \vert^\beta\int\limits_{\mathbb{R}^N} \vert z \vert^\alpha\vert w \vert^\beta dx.
		\end{eqnarray}
		Furthermore, we observe that 
		\begin{equation}
			F(1, 0) = 0, \,\,
			\frac{d}{dt}F(1, 0) = \Vert u \Vert^p - \lambda(p - 1) \Vert u \Vert_p^p = \lambda(2 - p) \Vert u \Vert_p^p > 0, (z,w) = (u,0).
		\end{equation}
		Now, by using the Implicit Function Theorem \cite{DRABEK}, there exists $\delta >0$ and a unique function $t : (-\delta, \delta) \to \mathbb{R}$ such that 
		\begin{equation}
			\frac{d}{dt}E_{_{\lambda}}(t(s)u, t(s)sv) = F(t(s), s) = 0, \,\, - \delta < s < \delta.
		\end{equation}
		In other words, we obtain that $(t(s) u, t(s)sv) \in \mathcal{N}_{_{\lambda}}$ holds for each $s \in (-\delta, \delta)$. Furthermore, assuming that $\delta > 0$ is small enough, we can use the fact that $s \mapsto \frac{d}{dt}F(t(s), s) $ is continuous proving that  that $\frac{d}{dt}F(t(s), s) > 0$ holds for each $s \in (-\delta, \delta)$. 
		Therefore, we infer that $(t(s) u, t(s)sv) \in \mathcal{N}^+_{_{\lambda}}$ holds for each  $s \in (-\delta, \delta)$. Furthermore, we mention that $t(s) \to 1$ 
		as $s \to 0$. Now, we define the auxiliary function  $G : (-\delta, \delta) \to \mathbb{R}$ of $C^1$ class given by $G(s) = E_{_{\lambda}}(t(s) u, t(s) s v), s \in (-\delta, \delta)$. 
		In particular, we know that $G(0) = E_{_{\lambda}}(u, 0) = C_{\mathcal{N}^+_{\lambda}}$. Moreover, we observe that  
		\begin{eqnarray*}
			G'(s)  & = & E'_{_{\lambda}}(t(s)u, t(s)sv)(t'(s)u, [ t'(s)s + t(s)] v)\\
			& = & \frac{t'(s)}{t(s)}E'_{_{\lambda}}(t(s)u, t(s)sv)(t(s)u, t(s)sv) + t(s) E'_{_{\lambda}}(t(s)u, t(s)sv)(0, v), s \in (-\delta, \delta).
		\end{eqnarray*}
		Notice also that $E'_{_{\lambda}}(t(s)u, t(s)sv)(t(s)u, t(s)sv) = 0$ holds for each  $s \in (-\delta, \delta)$. 
		Hence, we obtain the following identity
		\begin{eqnarray*}
			E'_{_{\lambda}}(t(s)u,t(s)s v)(0, v)& = &\left<(t(s)u,t(s)s v),(0, v)\right>  - \lambda \int\limits_{\mathbb{R}^N} \vert t(s)sv \vert^{q-2}t(s)svv dx  \\  &-& \frac{\theta}{\alpha + \beta}\beta \int_{\mathbb{R}^N} \vert t(s)u\vert^{\alpha} \vert t(s)sv\vert^{\beta-2}t(s)sv v dx.
		\end{eqnarray*}
		In particular, we obtain that
		\begin{eqnarray*}
			E'_{_{\lambda}}(t(s)u,t(s)s v)(0, v)
			& \leq &  t(s)s \Vert v \Vert^2  - \lambda (t(s)s)^{q - 1} \Vert v \Vert^q_q .
		\end{eqnarray*}
		As a consequence, we infer that 
		\begin{eqnarray*}
			G'(s)  & = & t(s) E'_{_{\lambda}}(t(s)u, t(s)sv)(0, v) \le t(s)\left (t(s)s \Vert v \Vert^2  - \lambda (t(s)s)^{q - 1} \Vert v \Vert^q_q\right) \nonumber \\
			& =&  t(s) (t(s)s)^{q - 1}\left [
			(t(s)s)^{2 - q} \Vert v \Vert^2 - \lambda \Vert v \Vert _q^q \right ].
		\end{eqnarray*}
		Since $q \in (1, 2)$ we obtain that $G'(s) < 0$ holds true for each $s \in (0, \delta)$ where $\delta > 0$ is small enough. Here was used the fact that $t(s) \to 1$ as $s \to 0$. Therefore, $s \mapsto G(s)$ is decreasing in the interval $[0, \delta)$. The last statement implies also that $G(s) < G(0) = C_{\mathcal{N}^+_{\lambda}}$. Hence, 
		we know that $(t(s)u, t(s)sv) \in \mathcal{N}^+_{\lambda}$ where $E_{_{\lambda}}(t(s)u, t(s)sv) < C_{\mathcal{N}^+_{\lambda}}$. This lead us to a contradiction proving that $(u,0)$ is not the infimum for the functional $E_\lambda$ restricted to the set $\mathcal{N}^+_\lambda$. Analogously, we infer that the function $(0, v)$ is not the infimum for the energy functional $E_\lambda$ restricted to $\mathcal{N}^+_\lambda$.  
		Therefore,  we see that $(u, v) \in \mathcal{A}$ is now satisfied for any minimizer for the functional $E_\lambda$ restricted to $\mathcal{A}\cap \mathcal{N}_\lambda^+$.  This ends the proof.
	\end{proof}
	
	\begin{prop}\label{regular}
		Suppose $(P)$, ($V_0$) and ($V_1$). Then any weak solution $(u, v) \in \mathcal{A}$ for the System (\ref{sistema Principal}) belongs to $C^{0, \alpha}(\mathbb{R}^N) \times C^{0, \alpha}(\mathbb{R}^N)$ for some $\alpha \in (0, 1).$
	\end{prop}
	\begin{proof}
		The proof follows the same ideas discussed in \cite[Proposition 4.1]{Regularizado} and \cite[Chapter 5]{DipierroRegul}. We will omit the details.
	\end{proof} 
	\begin{prop}
		Suppose $(P)$, ($V_0$) and ($V_1$). Let $(u_k, v_k)$ be a minimizer sequence for the energy functional $E_\lambda$ restricted to $\mathcal{N}^-_\lambda$. Then there exists a non-negative minimizer sequence for the energy functional $E_\lambda$ restricted to $\mathcal{N}^-_\lambda$.
	\end{prop}
	\begin{proof}
		Let $(u_k, v_k)$ be a minimizer sequence for the energy functional $E_\lambda$ restricted to  $\mathcal{N}^-_\lambda$. In other words, we have that $E_{_\lambda}(u_k, v_k) = C_{\mathcal{N}^-_{\lambda}} + o_k(1)$. It is not hard to verify that $(\vert u_k \vert , \vert v_k \vert) \in \mathcal{A}$. Here was used the fact that  $$\int\limits_{\mathbb{R}^N} \vert \vert u_k \vert \vert^\alpha\vert \vert v_k \vert \vert^\beta dx  = \int\limits_{\mathbb{R}^N} \vert u_k \vert^\alpha\vert v_k\vert^\beta dx > 0.$$
		Now, by using Proposition \ref{tn-,tn+}, there exists $t_n^-(\vert u_k \vert, \vert v_k \vert)$ such that $t_n^-(\vert u_k \vert, \vert v_k \vert)(\vert u_k \vert, \vert v_k \vert) \in \mathcal{N}^-_\lambda$. It is important to mention that $[\vert u_k \vert] \le [u_k]$ and $[\vert v_k \vert] \le [v_k]$. As a consequence, $$E'_{_\lambda}(t(\vert u_k \vert, \vert v_k \vert))(t(\vert u_k \vert, \vert v_k \vert )) \le E'_{_\lambda}(t(u_k, v_k))(t(u_k, v_k))$$ 
		holds for each $t > 0$. In particular, for $t = t_n^-(\vert u_k \vert, \vert v_k \vert )$, we deduce that 
		$$E'_{_\lambda}(t_n^- (\vert u_k \vert, \vert v_k \vert)(u_k, v_k))(t_n^- (\vert u_k \vert, \vert v_k \vert)(u_k, v_k)) \ge 0.$$
		Here we recall also that $t \mapsto E_{_\lambda}(tu_k, tv_k)$ is an increasing function in the interval $[t_n^+(u_k, v_k),t_n^-(u_k, v_k)]$.  As a consequence, $t_n^- (\vert u_k \vert, \vert v_k \vert) \in [t_n^+(u_k, v_k),t_n^-(u_k, v_k)]$. In particular, we obtain that \begin{eqnarray}
			C_{\mathcal{N}^-_\lambda} &\le& E_{_\lambda}(t_n^-(\vert u_k \vert, \vert v_k \vert)(\vert u_k\vert,\vert v_k \vert)) \le E_{_\lambda}(t_n^-(\vert u_k \vert, \vert v_k \vert)( u_k, v_k)) \nonumber \\
			&\le& \max\limits_{t \ge t_n^+(u_k, v_k)}E_{_\lambda}(t u_k, t v_k) = E_{_\lambda}(u_k, v_k) = C_{\mathcal{N}^-_\lambda} + o_k(1). 
		\end{eqnarray}
		Hence, $E_{_\lambda}(t_n^-(\vert u_k \vert, \vert v_k \vert)(\vert u_k\vert,\vert v_k \vert) = C_{\mathcal{N}^-_\lambda} + o_k(1)$. As a consequence, $(t_n^-(\vert u_k \vert, \vert v_k \vert)(\vert u_k\vert,\vert v_k \vert))$ is a non-negative minimizer sequence for the energy functional $E_\lambda$ restricted to $\mathcal{N}^-_\lambda$. This finishes the proof.     
	\end{proof}
	\begin{prop}\label{mini nao negat}
		Suppose $(P)$, ($V_0$) and ($V_1$). Let $(u_k, v_k)$ a minimizer sequence for the energy functional $E_\lambda$ restricted to $\mathcal{N}^+_\lambda \cap \mathcal{A}$. Then there exists a non-negative minimizer sequence for the energy functional $E_\lambda$ restricted to $\mathcal{N}^+_\lambda \cap \mathcal{A}$.
	\end{prop}
	\begin{proof}
		Let $(u_k, v_k)$ be a minimizer sequence for the energy functional $E_\lambda$ restricted to  $\mathcal{N}^+_\lambda\cap \mathcal{A}$. It is important to emphasize that $t_n^+(u_k, v_k) = 1$ holds for each $k \in \mathbb{N}$. Once again, we consider the sequence $(\vert u_k \vert, \vert v_k\vert)$. Now, by using Proposition \ref{tn-,tn+}, there exists $t_n^+(\vert u_k \vert, \vert v_k \vert)$ such that $t_n^+(\vert u_k \vert, \vert v_k \vert)(\vert u_k \vert, \vert v_k \vert) \in \mathcal{N}^+_\lambda$. Recall also that $\mathcal{A}$ is an open cone set. As a consequence, $t_n^+(\vert u_k \vert, \vert v_k \vert)(\vert u_k \vert, \vert v_k \vert) \in \mathcal{A}$. Therefore, we obtain that $t_n^+(\vert u_k \vert, \vert v_k \vert)(\vert u_k \vert, \vert v_k \vert) \in \mathcal{N}^+_\lambda \cap \mathcal{A}$. Now we claim that $C_{\mathcal{N}^+_\lambda \cap \mathcal{A}} = E_{_\lambda}(t_n^+(\vert u_k \vert, \vert v_k \vert)(\vert u_k \vert, \vert v_k \vert))  + o_k(1)$. The proof for this claim follows using some fine estimates. Firstly, we observe that 
		$$E'_{_\lambda}(t\vert u_k \vert, t\vert v_k \vert)(t\vert u_k \vert, t\vert v_k \vert) \le E'_{_\lambda}(t u_k, t v_k)(t u_k , t v_k ), \,\, t > 0.$$
		In particular, for each $t \in (0, 1)$, we obtain also that 
		$$E'_{_\lambda}(t\vert u_k \vert, t\vert v_k \vert)(t\vert u_k \vert, t\vert v_k \vert) \le 0.$$
		The last assertion implies that $t_n^+(\vert u_k \vert, \vert v_k \vert) \ge t_n^+(u_k, v_k) = 1$. It is important to stress that $t \mapsto E_{_\lambda}(t \vert u_k \vert, t \vert v_k \vert)$ is decreasing for $t \in (0, t_n^+(\vert u_k \vert, \vert v_k \vert)$. Under these conditions, we infer that 
		$$C_{\mathcal{N}^+_\lambda \cap \mathcal{A}} \leq E_{_\lambda}(t_n^+(\vert u_k \vert, \vert v_k \vert) (|u_k|, |v_k|)) \le E_{_\lambda}(\vert u_k \vert, \vert v_k \vert ) \le E_{_\lambda}( u_k, v_k) = C_{\mathcal{N}^+_\lambda \cap \mathcal{A}} + o_k(1).$$
		Hence, $C_{\mathcal{N}^+_\lambda \cap \mathcal{A}} = \lim\limits_{k \to \infty}E_{_\lambda}(t_n^+(\vert u_k \vert, \vert v_k \vert)(\vert u_k \vert, \vert v_k \vert))$ proving the claim. This ends the proof.
	\end{proof}
	
	\begin{prop}\label{sol positivas}
		Suppose $(P)$, ($V_0$) and ($V_1$). Assume that $\lambda \in (0, \lambda^*)$. Then the System (\ref{sistema Principal}) admits at least two positive weak solutions $(u, v)$ and $(z, w)$. 
	\end{prop}
	\begin{proof}
		Let $(u_k, v_k) \in \mathcal{N}^+_{\lambda} \cap \mathcal{A}$ be a minimizer sequence for the energy functional $E_{_\lambda}$ restricted to $\mathcal{N}^+_{\lambda} \cap \mathcal{A}$. It is easy to verify that  $(\vert u_k \vert, \vert v_k \vert) \in X$ and $(\vert u_k \vert, \vert v_k \vert) \in \mathcal{A}$.  In light of Proposition \ref{mini nao negat} we know that $(W_{k,1}, W_{k,2}) = ( t_n^+(\vert u_k \vert, \vert v_k \vert) \vert u_k\vert,\vert u_k \vert, \vert v_k \vert) \vert v_k\vert)$ is a non-negative minimizer sequence for the energy functional $E_{_\lambda}$ restricted $\mathcal{N}^+_{\lambda} \cap \mathcal{A}$. Furthermore,  by using Proposition \ref{soluçao N^+}, there exists $(W_1, W_2) \in X$ such that $$C_{\mathcal{N}^+_{\lambda} \cap \mathcal{A}} = E_{_\lambda}(W_1, W_2)$$ where $(W_{k,1}, W_{k,2}) \to (W_1, W_2)$ in $X$. In particular, $(W_{1}, W_{2})$ is a critical point of the energy functional. Notice also that $W_1, W_2 \ge 0$ in $\mathbb{R}^N$. According to \cite[Lemma 2,8]{Vambrozio2020} and Proposition \ref{regular} we conclude that $(W_1, W_2) \in (L^t(\mathbb{R}^N)\cap L^{\infty}(\mathbb{R}^N)) \times (L^t(\mathbb{R}^N)\cap L^{\infty}(\mathbb{R}^N))$ holds for each $t \in (2^*_s, \infty)$. Based on  \cite[Corollary 5.5 ]{AIalMoscSqua2016} and \cite[Lemma 2.8]{Vambrozio2020}, we conclude that $(W_1, W_2) \in (L^\infty(\mathbb{R}^N) \cap C^{0, \alpha}(\mathbb{R}^N)) \times (L^\infty(\mathbb{R}^N) \cap C^{0, \alpha}(\mathbb{R}^N))$. Now, we claim that $W_1 > 0$ and $W_2 > 0$ in $\mathbb{R}^N$. The proof for this claim follows arguing by contradiction. Assume that there exists $x_0 \in \mathbb{R}^N$ such that $W_1(x_0) = 0$. Let $B(x_0, r)$ be the open ball of radius $r > 0$ centered on $x_0$. It is not hard to verify that 
		\begin{equation}\left\{\begin{array}{lll}\label{principio máx forte} 
				(-\Delta)^sW_1 \le - V_1(x)W_1, \,\, x\in B(x_0, r), \\
				W_1 \ge 0  \mbox{ in } \mathbb{R}^N.
			\end{array}\right.
		\end{equation}
		According to strong maximum principle we obtain that $W_1 > 0$ in $\mathbb{R}^N$ or $W_1 = 0$ in $\mathbb{R}^N$, see \cite[Theorem 1.2]{LMdel2017}. Using the last assertion we deduce that $W_1 = 0$ in $B(x_0,r)$. Since $r > 0$ is arbitrary we infer that $W_1 = 0$ in $\mathbb{R}^N$. Hence, $(W_1, W_2) \notin \mathcal{A}$. However, by using the Proposition \ref{soluçao N^+}, we observe that the $(W_1, W_2) \in \mathcal{A}$.  This is a contradiction proving that $W_1$ is strictly positive in $\mathbb{R}^N$. Analogously, we obtain that $W_2$ is strictly positive in $\mathbb{R}^N$. As a consequence, $W_1, W_2$ are strictly positive in $\mathbb{R}^N$. Moreover, $(W_1,W_2) \in \mathcal{N}^+_{\lambda} \cap \mathcal{A}$ is a weak solution for the System  \eqref{sistema Principal}.
	\end{proof}

	\section{The proof of the main results}
	\noindent \textbf{The proof of Theorem \ref{teorem1}}
	Let  $\lambda \in (0, \lambda^*)$ be fixed. Due to the Propositions \ref{soluçao N^+} and \ref{sol positivas} there exists $(u, v) \in \mathcal{N}^+_{\lambda}\cap \mathcal{A}$ which is a critical point for the energy functional $E_{\lambda}$. As a consequence, $(u, v)$ is a positive solution for the System  \eqref{sistema Principal}.  
	Furthermore, $E_{\lambda}(u, v) = C_{\mathcal{N}^+_{\lambda} \cap \mathcal{A}} \le E_{\lambda}(t_n^+(u, v)(u, v)) = C_{\lambda} < 0$, see Proposition \ref{CN+ neg}.

	\noindent\textbf{The proof of Theorem \ref{teorem2}} \textbf{(i)} The main idea here is to prove that the minimization problem \eqref{CN-} admits at least one solution. In order to do that we shall apply the same ideas discussed in the proof of Proposition  \ref{soluçao N-}. Firstly, by using the Proposition \ref{CN+ neg}, there exists $(z, w) \in \mathcal{N}^-_\lambda$ such that $C_{\mathcal{N}^-_{\lambda} \cap \mathcal{A}} = E_{\lambda}(z, w)$. Here we recall also that $\mathcal{N}_\lambda^0 = \emptyset$ for each $\lambda \in (0 , \lambda^*)$. In particular, $(z, w) \in \mathcal{N}^-_\lambda$ is a critical point for the functional $E_\lambda$ which implies that $(z,w)$ is a weak solution to the System  \eqref{sistema Principal}. Notice also that $E''_{\lambda}(z, w)(z, w)^2 < 0$ holds true for each $\lambda \in (0 , \lambda^*)$. 
	
	\noindent \textbf{(ii)} Consider $\lambda \in (0, \lambda_*)$.  It is easy to verify that $t_n^-(z, w) = 1 > t_e(z, w)$. Hence, $$\lambda = R_n(z, w) = R_n(t_n^-(z, w)(z, w) < R_e(t_n^-(z, w)(z, w) = R_e(z, w).$$ 
	Thus, by using Remark \ref{obs rel Re e E}, we deduce that $E_{\lambda}(z, w) > 0$ and $C_{\mathcal{N}_{_\lambda}^-} =  E_{\lambda}(z, w) > 0$.\\
	\textbf{(iii)} For the case $\lambda = \lambda_*$ we obtain that $t_n^-(z, w) = t_e(z, w)$. The last identity implies that  $$\lambda = \lambda_*= \Lambda_e(z, w) = R_e(t_e(z, w)(z, w)) = R_e(t_n^-(z, w)(z, w)) = R_e(z, w).$$ As a consequence, by using Remark \ref{obs rel Re e E}, we deduce that $E_{\lambda}(z,w) = 0$.
	
	\noindent \textbf{(iv)} Let $\lambda \in (\lambda_*, \lambda^*)$ be fixed. Hence, for any fixed  $(u, v) \in \mathcal{A}$ we ensure the existence of two zeros for the equation  $Q_n(t) = \lambda$, see Proposition  \ref{tn-,tn+}. Furthermore, we know that $t_n^-(u, v)(u, v) \in \mathcal{N}^-_{\lambda}$. Thus, we deduce that  $\lambda_* \le \Lambda_e (u, v) = R_e(t_e(u, v)(u, v)).$ Notice also that $t_n^-(u, v) \in (0, t_e(u, v))$. Here we also observe that $R_e(t_n^-(u, v)(u, v)) < R_n(t_n^-(u, v)(u, v)) = \lambda$. Therefore, by using Remark \ref{obs rel Re e E}, we infer that  $E_{\lambda}(t_n^-(u, v)(u, v)) < 0$. As a consequence, $C_{\mathcal{N}^-_{_\lambda}} \le E_{\lambda}(t_n^-(u, v)(u, v)) < 0$. This ends the proof.

	\noindent \textbf{The proof of Theorem \ref{teorem3}} According to the Proposition \ref{soluçao N-} there exists $(u, v) \in \mathcal{N}^-_{\lambda}$ which is a critical point for the energy functional. Furthermore, by using Proposition \ref{soluçao N^+}, there exists $(z, w) \in \mathcal{N}^+_{\lambda} \cap \mathcal{A}$ which is a critical point for the energy functional. It is important to  observe that $(u, v) \neq (z, w)$ due to the fact that $\mathcal{N}^+_{\lambda} \cap \mathcal{N}^+_\lambda = \emptyset$. Moreover, by using Proposition \ref{sol positivas}, we also obtain that $u > 0, v > 0, z > 0$ and $w > 0$ in $\mathbb{R}^N$. This ends the proof. 
	
	\section{Appendix}
	In the present appendix we shall consider some useful results used in the present work. Firstly, we shall consider the following result:
	
	\begin{lem}\label{apendice2A}
		Suppose $(P)$, ($V_0$) e ($V_1$). Then we obtain that $2^{\frac{\eta - q}{\eta - 2}} \eta^{\frac{q - 2}{\eta - 2}}/q >1$ where 
		$\eta = \alpha + \beta$.
	\end{lem}
	\begin{proof}
		Firstly, we observe that $2^{\frac{\eta - q}{\eta - 2}} \eta^{\frac{q - 2}{\eta - 2}}/q >1$, is equivalent to prove that
		$$\left(\frac{\eta}{2}\right)^{q - 2} > \left(\frac{q}{2}\right)^{\eta - 2}.$$
		Define the follows auxiliary function $f: \mathbb{R}^+ \to \mathbb{R}$ given by 
		$$f(x) := (x - 2) \ln \left(\frac{\eta}{2}\right) - (\eta - 2) \ln \left(\frac{x}{2}\right).$$
		The main purpose here is to ensure that $f(x) > 0$ for each $x\in (1, 2)$. It is not hard to verify that $f''(x) = (\eta - 2)/x^2 > 0$ and $f(2) = f(\eta) = 0$. In particular, we infer that $f'(x) > 0$ for each $x > x_0 =  (\eta - 2) / \ln (\eta/2)$. Moreover, we also obtain that $f'(x) < 0$ for each $x < x_0$. It is important to stress that $(\eta - 2) / \ln (\eta/2) > 2$.  Hence, the function $f$ has a unique critical point $x_0 \in ( 2, \eta)$. Therefore, $f(x) > 0 $ for each $x \in (1, 2)$. This ends the proof.    
	\end{proof}
	\begin{lem}\label{apendice Maxwell}
		Consider the function $f: \mathbb{R} \to \mathbb{R}$ and $t_n > 0$ in such way that $f(t) = \frac{At^2 - Bt^\eta}{Ct^{p} + Dt^{q}}$ where $f(t_n) = \max\limits_{t>0} f(t)$ and $1 \le p < q < \eta$. Then $f$ has a single global maximum critical point.
	\end{lem}
	\begin{proof}
		Recall that $$f(t) = \frac{At^2 - Bt^\eta}{Ct^{p} + Dt^{q}}$$ and $f(t_n) = \max_{t>0} f(t)$ where $1 \le p < q < \eta$. In particular, we observe that 
		\begin{equation}\begin{array}{lll}\label{apendA1}
				\left\{\begin{array}{c}
					G(t):=Ct^{p} + Dt^{q} > 0 \ \ \forall t > 0 \\
					G'(t) = pC t^{p - 1} + qD t^{q - 1} , \ \ \forall t > 0
				\end{array}\right.
			\end{array}
		\end{equation}
		and
		\begin{equation}\begin{array}{lll}\label{apendA2}
				\left\{\begin{array}{c}
					H(t):=At^2 - Bt^{\eta} > 0 \ \ if \ \ 0 < t < \left\{\frac{A}{B} \right\}^{\frac{1}{\eta -2}} \\
					H'(t) = 2A t - \eta B t^{\eta -1} > 0 \ \ if \ \  0 < t < \left\{\frac{2A}{\eta B} \right\}^{\frac{1}{\eta -2}} \\
					H'(t) = 2A t - \eta B t^{\eta -1} < 0 \ \ if \ \ t > \left\{\frac{2A}{\eta B} \right\}^{\frac{1}{\eta -2}} \\
				\end{array}\right.
			\end{array}
		\end{equation}
		Notice also that $f(t)G(t) = H(t)$ and 
		$f'(t)G(t) + f(t)G'(t) = H'(t)$. Therefore, for $t > 0$ and $t \ne \left\{A/B \right\}^{\frac{1}{\eta -2}}$, we obtain that
		$$f'(t) = 0 \ \ \Longleftrightarrow \ \ f(t) = \frac{H'(t)}{G'(t)} \ \  \Longleftrightarrow \ \ \frac{H(t)}{G(t)} = \frac{H'(t)}{G'(t)} \ \  \Longleftrightarrow \ \  \frac{tH'(t)}{H(t)} = \frac{tG'(t)}{G(t)}.$$
		Consider the function $\mathcal{G} : \mathbb{R} \to \mathbb{R}$ given by 
		$$\mathcal{G}(t):= \frac{tG'(t)}{G(t)} = \frac{pCt^{p} + q Dt^{q}}{Ct^{p} + Dt^{q}} = \frac{pCt^{p} + pDt^{q} +(q - p) Dt^{q}}{Ct^{p} + Dt^{q}} = q + \frac{(p - q)Dt^q}{Ct^{p} + Dt^{q}}.$$
		It is easy to verify that $p < \mathcal{G}(t) < q$ for all $t > 0$. Therefore, we obtain that 
		$$\mathcal{G}'(t) = (p - q)^2CD\frac{t^{q + p - 1}}{(Ct^{p} + Dt{q})^2}.$$ 
		Hence, $\mathcal{G}(t)$ is strictly increasing. Notice also that the function $\mathcal{H}: \mathbb{R} \to \mathbb{R}$ given by 
		$$\mathcal{H}(t):= \frac{tH'(t)}{H(t)} = \frac{2At^2 - \eta Bt^\eta}{At^2 - Bt^\eta} = \frac{2At^2 - 2Bt^\eta - (\eta - 2)Bt^\eta}{At^2 - Bt^\eta} = 2 - \frac{(\eta - 2)Bt^\eta}{At^2 - Bt^\eta}$$
		satisfies $$\mathcal{H}'(t) = - (\eta - 2)^2 AB\frac{t^{\eta + 1}}{(At^2 - B t^\eta)^2} < 0.$$ Furthermore, we observe that
		\begin{equation}\label{apendA3}
			\lim_{t \to 0^+} \mathcal{H}(t) = 2, 
			\lim_{t \to \infty} \mathcal{H}(t) = \eta,
		\end{equation}
		\begin{equation}\label{apendA4}
			\lim_{t \to \left(\left[\frac{A}{B}  \right]^\frac{1}{\eta - 1}\right)^-} \mathcal{H}(t) = - \infty,
			\lim_{t \to \left(\left[\frac{A}{B}  \right]^\frac{1}{\eta - 1}\right)^+} \mathcal{H}(t) =  \infty.
		\end{equation}
		It is not hard to verify also that 
		\begin{equation}\label{apendA4}
			\mathcal{H}(t) < 0, \,\, \mbox{for each} \,\, \left[\frac{2A}{\eta B}  \right]^\frac{1}{\eta - 2} < t < \left[\frac{A}{B}  \right]^\frac{1}{\eta - 2} \,\, \mbox{and} \,\,
			\mathcal{H}(t) > \eta, \, \, \mbox{for each} \,\,  t > \left[\frac{A}{B}  \right]^\frac{1}{\eta - 2}.
		\end{equation}
		Moreover, we obtain that $p < \mathcal{G}(t) < q$. Hence, we deduce that the functions $\mathcal{H}$ and $\mathcal{G}$ are equal only once on the entire real line. Furthermore, $\mathcal{H}$ and $\mathcal{G}$ are equal in some point $0 < t < \left[2A/\eta B\right]^\frac{1}{\eta - 2}$. In other words, there exists a single $t \in \left(0, \left[2A/\eta B\right]^\frac{1}{\eta - 2}\right)$ such that $$\frac{tH'(t)}{H(t)} = \frac{tG'(t)}{G(t)}.$$ This assertion shows that $f$ has a single global maximum point. This ends the proof. 
	\end{proof}
	
	Now, we shall consider a basic result to apply in the function $f$ described in the Lemma \ref{apendice Maxwell}. More specifically, using a standard argument, we consider the following basic result:  
	
	\begin{lem}\label{ex3.18Elon analise reta} 
		Let $a_1$, $b_1$, $a_2$, $b_2$ be real numbers where $b_1$ e $b_2$ are positive. Then $\frac{a_1 + a_2}{b_1 + b_2}$ is between the smallest and the largest of elements $\frac{a_1}{b_1}$ and $\frac{a_2}{b_2}$, i.e., we have
		$$\min\left\{\frac{a_1}{b_1}, \frac{a_2}{b_2}\right\} \le \frac{a_1 + a_2}{b_1 + b_2} \le \max\left\{\frac{a_1}{b_1}, \frac{a_2}{b_2}\right\}.$$
	\end{lem}

	%\begin{acknowledgment}
	%	The authors would like to thank the referees for careful reading the manuscript and valuable comments and suggestions.
	%\end{acknowledgment}
	
	%%%%%%%%%%%%%%%%%%%%%%%%%%%%%%%%%%%%%%%%%%%%%%%%%%%%%%%%%%%%%%%%%%%%%%%%%%%%%%%%%%%%%%%%%%%%%%%%%%%%%%%%%%%%%%%%%%%%%%%%%%%%%%%%%%%%%%%%%%%%%%%%%%%%%%%%%%%%%%%%%%%%%%%%%%%%%%%%%%%%%%%
	%                                                                             REFERENCES
	%%%%%%%%%%%%%%%%%%%%%%%%%%%%%%%%%%%%%%%%%%%%%%%%%%%%%%%%%%%%%%%%%%%%%%%%%%%%%%%%%%%%%%%%%%%%%%%%%%%%%%%%%%%%%%%%%%%%%%%%%%%%%%%%%%%%%%%%%%%%%%%%%%%%%%%%%%%%%%%%%%%%%%%%%%%%%%%%%%%%%%%

	\section{\textbf{Declarations}}
	
	\textbf{Ethical Approval }
	
	It is not applicable.

	\textbf{Competing interests}
	
	There are no competing interests.

	\textbf{Authors' contributions}

	All authors wrote and reviewed this paper.

	\textbf{Funding}

	CNPq/Brazil with grant 309026/2020-2 and IFG with grant 23744.001007/2019-69.

	\textbf{Availability of data and materials}

	All the data can be accessed from this article.

\end{document}